\documentclass[12pt]{amsart}
\usepackage{amssymb,verbatim,url,mathrsfs,tikz,extarrows,enumerate}
\usepackage[pdfpagelabels]{hyperref}
\usepackage[nameinlink]{cleveref}
\usetikzlibrary{matrix,arrows}
\numberwithin{equation}{subsection}
\newtheorem{theorem}[subsection]{Theorem}
\newtheorem{corollary}[subsection]{Corollary}
\newtheorem{proposition}[subsection]{Proposition}
\newtheorem{lemma}[subsection]{Lemma}

\newtheoremstyle{mystyle}
  {}
  {}
  {}
  {0pt}
  {}
  {.}
  { }
  {\textbf{\thmname{#1}\thmnumber{ #2}}\thmnote{ (#3)}}
\theoremstyle{mystyle}
\newtheorem{definition}[subsection]{Definition}

\newtheorem{note}[subsection]{Note}

\newtheorem{remark}[subsection]{Remark}

\newtheoremstyle{pgstyle}
  {}
  {}
  {}
  {}
  {}
  {.}
  { }
  {\textbf{\thmname{#1}\thmnumber{#2}}\thmnote{ (#3)}}
\theoremstyle{pgstyle}
\newtheorem{pg}[subsection]{}

\newcommand{\mc}[1]{\mathcal{#1}}
\newcommand{\ms}[1]{\mathscr{#1}}

\newcommand{\mr}[1]{\mathrm{#1}}
\newcommand{\mb}[1]{\mathbf{#1}}
\newcommand{\on}[1]{\operatorname{#1}}

\newcommand{\et}{\mathrm{\acute et}}
\DeclareMathOperator{\Aut}{Aut}
\DeclareMathOperator{\Br}{Br}
\DeclareMathOperator{\GL}{GL}
\DeclareMathOperator{\SL}{SL}
\DeclareMathOperator{\PGL}{PGL}
\DeclareMathOperator{\Spec}{Spec}
\DeclareMathOperator{\Proj}{Proj}

\DeclareMathOperator{\Hom}{Hom}
\DeclareMathOperator{\id}{id}
\DeclareMathOperator{\Pic}{Pic}
\DeclareMathOperator{\Mat}{Mat}
\DeclareMathOperator{\Mor}{Mor}

\DeclareMathOperator{\im}{im}
\def\A{\mathbb{A}}
\def\Z{\mathbb{Z}}
\def\G{\mathbb{G}}
\def\H{\mathrm{H}}

\def\F{\mathbb{F}}
\def\N{\mathbb{N}}

\begin{document}
\title[$\Br \ms{M}_{1,1,k} = \Z/(2)$ for $k = \overline{k}$ and $\on{char} k = 2$]{The Brauer group of $\ms{M}_{1,1}$ over algebraically closed fields of characteristic $2$}
\author{Minseon Shin}
\email{shinms@math.berkeley.edu}
\urladdr{\url{http://math.berkeley.edu/~shinms}}
\begin{abstract} We prove that the Brauer group of the moduli stack of elliptic curves $\ms{M}_{1,1,k}$ over an algebraically closed field $k$ of characteristic $2$ is isomorphic to $\Z/(2)$. We also compute the Brauer group of $\ms{M}_{1,1,k}$ where $k$ is a finite field of characteristic $2$. \end{abstract}
\date{\today}
\maketitle

\section{Introduction} \label{sec-01}

Let $\ms{M}_{1,1,\Z}$ denote the moduli stack of elliptic curves over $\Z$. For any scheme $S$, we denote by $\ms{M}_{1,1,S} := S \times_{\Z} \ms{M}_{1,1,\Z}$ the restriction of $\ms{M}_{1,1,\Z}$ to the category of schemes over $S$. \par Antieau and Meier \cite[11.2]{ANTIEAU-MEIER-TBGOTMSOEC} computed the Brauer group $\Br \ms{M}_{1,1,S}$ for various base schemes $S$, and in particular proved that for any algebraically closed field $k$ of characteristic not $2$ the Brauer group $\Br \ms{M}_{1,1,k}$ is trivial. The purpose of this note is to compute $\Br \ms{M}_{1,1,k}$ in the characteristic $2$ case. This then completes the calculation of $\Br \ms{M}_{1,1,k}$ over algebraically closed fields $k$. We summarize the result in the following theorem.

\begin{theorem}[{\cite[11.2]{ANTIEAU-MEIER-TBGOTMSOEC}} in $\on{char} k \ne 2$] \label{48} Let $k$ be an algebraically closed field. Then $\Br \ms{M}_{1,1,k}$ is $0$ unless $\on{char} k = 2$, in which case $\Br \ms{M}_{1,1,k} = \Z/(2)$. \end{theorem}

To prove the theorem, we calculate the cohomology groups $\H^{2}_{\et}(\ms{M}_{1,1,k} , \mu_{n})$ for varying $n$. There are essentially two ways to approach this calculation: (1) using the coarse moduli space; (2) using a presentation of $\ms{M}_{1,1,k}$ as a quotient stack. In this paper we give a new proof of the Antieau-Meier result using approach (1), and calculate in characteristic 2 using approach (2). 

We also compute the Brauer group of $\ms{M}_{1,1,k}$ where $k$ is a finite field of characteristic $2$:

\begin{theorem} \label{20170322-02} Let $k$ be a finite field of characteristic $2$. Then \[ \Br \ms{M}_{1,1,k} = \begin{cases} \Z/(12) \oplus \Z/(2) &\text{if }x^{2}+x+1 \text{ has a root in }k \\ \Z/(24) &\text{otherwise.} \end{cases} \] \end{theorem} 

An outline of the paper is as follows.

In \Cref{sec-05} we state definitions and recall general facts about the Brauer group of algebraic stacks.

In \Cref{sec-04} we record some general remarks regarding $\Br \ms{M}_{1,1,S}$. We show that if $S$ is a quasi-compact scheme admitting an ample line bundle and if at least one prime is invertible on $S$, then $\Br \ms{M}_{1,1,S} \simeq \Br' \ms{M}_{1,1,S}$. The restriction of $\ms{M}_{1,1,\Z}$ to the dense open substack of elliptic curves $E/S$ with $j$-invariant $j(E) \in \Gamma(S,\mc{O}_{S})$ for which $j(E)$ and $j(E)-1728$ are invertible is a trivial $\Z/(2)$-gerbe over the coarse space $\A_{\Z}^{1} \setminus \{0,1728\}$, and we use this fact to conclude that $\Br \ms{M}_{1,1,k}$ is a subgroup of $\Z/(2) \oplus \Z/(2)$ for an algebraically closed field $k$ of arbitrary characteristic. 

In \Cref{sec-02} we give a second proof of Antieau and Meier's result above (that $\Br \ms{M}_{1,1,k} = 0$ if $k = \overline{k}$ and $\on{char} k \ne 2$). Using a d\'evissage argument, we study the relationship between the cohomology of $\mu_{n}$ on the stack $\ms{M}_{1,1,k}$ and on $\A_{k}^{1}$, in terms of the stabilizer groups of elliptic curves with $j$-invariant $0,1728 \in \A_{k}^{1}$. This may be of independent interest for computing the Brauer groups of other separated Deligne-Mumford stacks whose coarse moduli space is a smooth curve over an algebraically closed field with vanishing Picard group.

In \Cref{sec-03} we prove \Cref{48} and \Cref{20170322-02}. Antieau and Meier suggest in \cite[11.3]{ANTIEAU-MEIER-TBGOTMSOEC} that the characteristic $2$ case can be settled using the $\GL_{2}(\Z/(3))$-cover $Y(3) \to \ms{M}_{1,1,k}$, where $Y(3)$ denotes the moduli stack of elliptic curves with full level 3 structure, and indeed we use this presentation $\ms{M}_{1,1,k} \simeq [Y(3) / \GL_{2}(\Z/(3))]$ as a global quotient stack to show that its Brauer group is in fact nonzero. We use the ``Hesse presentation'' of $Y(3)$ as in \cite{FULTON-OLSSON-PICARD}; it is shown in \Cref{sec-08} that this presentation coincides with the usual Weierstrass presentation as in \cite{KATZ-MAZUR-AMOEC}. The cohomological descent spectral sequence associated to the covering $Y(3) \to \ms{M}_{1,1,k}$ reduces our task to a computation of the first group cohomology of a $6$-dimensional representation of $\GL_{2}(\Z/(3))$ over $\F_{2}$.

\begin{pg}[Acknowledgements] I thank my advisor Martin Olsson for suggesting this research topic and for his generosity in sharing his ideas. I am also grateful to Benjamin Antieau, Siddharth Mathur, and Lennart Meier for helpful discussions. During this project, I received support from the Raymond H. Sciobereti Fellowship. \end{pg}

\section{The Brauer group of algebraic stacks} \label{sec-05}

Let $(X,\mc{O}_{X})$ be a locally ringed site \cite[V, \S4]{GIRAUD-CN}, \cite[04EU]{STACKS-PROJECT}. For any quasi-coherent $\mc{O}_{X}$-module $\ms{E}$, we set $\GL(\ms{E}) := \underline{\Aut}_{\mc{O}_{X}\text{-mod}}(\ms{E})$ and let $\PGL(\ms{E})$ be the sheaf quotient of $\GL(\ms{E})$ by $\G_{m,X}$ via the diagonal embedding. We denote $\GL_{n}(\mc{O}_{X}) := \GL(\mc{O}_{X}^{\oplus n})$ and $\PGL_{n}(\mc{O}_{X}) := \PGL(\mc{O}_{X}^{\oplus n})$. A basic fact about these groups is the Skolem-Noether theorem, which states that the morphism \begin{align*} \PGL_{n}(\mc{O}_{X}) \to \underline{\Aut}_{\mc{O}_{X}\textnormal{-alg}}(\Mat_{n \times n}(\mc{O}_{X})) \end{align*} is an isomorphism (see \cite[V.4.1]{GIRAUD-CN}).

\begin{definition}[Azumaya algebras] \label{19} \cite[\S2]{Gro68a}, \cite[V, \S4]{GIRAUD-CN} Let $(X,\mc{O}_{X})$ be a locally ringed site. An \emph{Azumaya $\mc{O}_{X}$-algebra} is a quasi-coherent (non-commutative, unital) $\mc{O}_{X}$-algebra $\ms{A}$ such that there exists a covering $\{X_{i} \to X\}_{i \in I}$, positive integers $n_{i}$, and $\mc{O}_{X_{i}}$-algebra isomorphisms $\ms{A}|_{X_{i}} \simeq \Mat_{n_{i} \times n_{i}}(\mc{O}_{X_{i}})$. \par Two Azumaya algebras $\ms{A}_{1}$ and $\ms{A}_{2}$ are \emph{Morita equivalent} if there exist finite type locally free $\mc{O}_{X}$-modules $\mc{E}_{1}$ and $\mc{E}_{2}$, everywhere of positive rank, and an isomorphism \[ \ms{A}_{1} \otimes_{\mc{O}_{X}} \underline{\on{End}}_{\mc{O}_{X}\textnormal{-mod}}(\mc{E}_{1}) \simeq \ms{A}_{2} \otimes_{\mc{O}_{X}} \underline{\on{End}}_{\mc{O}_{X}\textnormal{-mod}}(\mc{E}_{2}) \] of $\mc{O}_{X}$-algebras. Under tensor product of Azumaya algebras, Morita equivalence classes of Azumaya algebras form an abelian group $\Br X$ called the \emph{(Azumaya) Brauer group} of $X$ in which $[\ms{A}]^{-1} = [\ms{A}^{\on{op}}]$ and the identity element is the class of trivial Azumaya algebras $[\underline{\on{End}}_{\mc{O}_{X}\textnormal{-mod}}(\mc{E})]$. \end{definition}

\begin{definition}[Gerbe of trivializations] \label{18} \cite[IV, \S4.2]{GIRAUD-CN}, \cite[12.3.5]{OLSSON} There is a natural way to associate, to every Azumaya $\mc{O}_{X}$-algebra $\ms{A}$, a $\G_{m,X}$-gerbe $\mc{G}_{\ms{A}}$ called the \emph{gerbe of trivializations of $\ms{A}$}. An object of $\mc{G}_{\ms{A}}$ is a triple \[ (U,\mc{E},\sigma) \] consisting of an object $U \in X$, a finite type locally free $\mc{O}_{U}$-module $\mc{E}$ (necessarily everywhere positive rank), and an isomorphism $\sigma : \underline{\on{End}}_{\mc{O}_{U}\textnormal{-mod}}(\mc{E}) \to \ms{A}|_{U}$ of $\mc{O}_{U}$-algebras. A morphism \[ (f,f^{\sharp}) : (U_{1},\mc{E}_{1},\sigma_{1}) \to (U_{2},\mc{E}_{2},\sigma_{2}) \] consists of a morphism $f \in \Mor_{X}(U_{1},U_{2})$ and an isomorphism $f^{\sharp} : f^{\ast}\mc{E}_{2} \to \mc{E}_{1}$ of $\mc{O}_{U_{1}}$-modules such that $\sigma_{2} = \sigma_{1} \circ \rho_{f^{\sharp}}$ where $\rho_{f^{\sharp}}$ denotes conjugation by $f^{\sharp}$. For any object $(U,\mc{E},\sigma) \in \mc{G}_{\ms{A}}$, there is a canonical injection \[ \iota_{(U,\mc{E},\sigma)} : \G_{m,U} \to \underline{\on{Aut}}_{(U,\mc{E},\sigma)} \] of sheaves on $X/U$, sending $u \mapsto (\id_{U},u)$; this is in fact an isomorphism, since if $(\id_{U},f^{\sharp}) \in \on{Aut}_{\mc{G}_{\ms{A}}(U)}((U,\mc{E},\sigma))$ then $f^{\sharp} \in Z(\on{End}_{\mc{O}_{U}\textnormal{-mod}}(\mc{E}))$, which coincides with $\mc{O}_{U}$ since $Z(\Mat_{n \times n}(A)) = A$ for any commutative, unital ring $A$. \par By the Skolem-Noether theorem, any two local trivializations of $\ms{A}$ are locally related by an automorphism of the trivializing vector bundle $\mc{E}$, i.e. any two objects of $\mc{G}_{\ms{A}}$ are locally isomorphic. Furthermore, according to the definition, an Azumaya algebra is locally trivial, i.e. for any $U \in X$ there exists a covering $\{U_{i} \to U\}$ such that the fiber category $\mc{G}_{\ms{A}}(U_{i})$ is nonempty. These considerations show that $\mc{G}_{\ms{A}}$ is a $\G_{m,X}$-gerbe. \par The assignment $\ms{A} \mapsto \mc{G}_{\ms{A}}$ induces a group homomorphism \begin{align} \label{eqn-02} \alpha_{X}' : \Br X \to \H^{2}(X,\G_{m,X}) \end{align} which is injective since a $\G_{m,X}$-gerbe $\mc{G}$ is trivial if and only if $\mc{G}(X)$ is nonempty. \par For a morphism \[ (f,f^{\sharp}) : (X,\mc{O}_{X}) \to (Y,\mc{O}_{Y}) \] of locally ringed sites, the diagram \begin{equation} \label{eqn-10} \begin{tikzpicture}[>=angle 90, baseline=(current bounding box.center)] 
\matrix[matrix of math nodes,row sep=2em, column sep=2em, text height=1.5ex, text depth=0.25ex] { 
|[name=11]| \Br X & |[name=12]| \H^{2}(X,\G_{m,X}) \\ 
|[name=21]| \Br Y & |[name=22]| \H^{2}(Y,\G_{m,Y}) \\
}; 
\draw[->,font=\scriptsize]
(11) edge node[above=0pt] {$\alpha_{X}'$} (12) (21) edge node[below=0pt] {$\alpha_{Y}'$} (22) (21) edge node[left=-1pt] {$f^{\ast}$} (11) (22) edge node[right=-1pt] {$f^{\ast}$} (12); \end{tikzpicture} \end{equation} is commutative. \end{definition}

\begin{lemma} \label{03} Let $\ms{X}$ be a $\G_{m,X}$-gerbe over a locally ringed site $X$. The class $[\ms{X}] \in \H^{2}(X,\G_{m,X})$ is in the image of $\alpha_{X}'$ if and only if $\ms{X}$ admits a $1$-twisted finite locally free sheaf of everywhere positive rank. \end{lemma}

The usual proof (c.f. \cite[2.14]{DEJONG-GABBER}, \cite[3.1.2.1]{LIEBLICH-TSAPIP}, \cite[12.3.11]{OLSSON}) of \Cref{03} applies more generally to the case of $\G_{m}$-gerbes over an arbitrary locally ringed site.

We will only consider locally ringed sites $(X,\mc{O}_{X})$ whose underlying site $X$ is quasi-compact \cite[090G]{STACKS-PROJECT}. For such $X$, the Brauer group $\Br X$ is a torsion group.

\begin{definition} The torsion subgroup of $\H^{2}(X,\G_{m,X})$, denoted $\Br' X$, is called the \emph{cohomological Brauer group} and the restriction \begin{align} \alpha_{X} : \Br X \to \Br' X \end{align} of $\alpha_{X}'$ to $\Br' X$ is called the \emph{Brauer map}. \end{definition}

We will consider algebraic stacks using the \'etale topology except in \Cref{sec-03} (the case of characteristic $2$) in which we will require the flat topology.

Surjectivity of the Brauer map may be checked on a finite flat surjective covering (c.f. \cite[II, Lemma 4]{GABBER-THESIS}, \cite[2.15]{DEJONG-GABBER}, \cite[3.1.3.5]{LIEBLICH-TSAPIP}):

\begin{proposition} \label{01} Let $f : X \to Y$ be a finitely presented, finite, flat, surjective morphism of algebraic stacks. A class $\beta \in \H^{2}(Y,\G_{m,Y})$ is in the image of $\alpha_{Y}'$ if and only if its pullback $f^{\ast}\beta \in \H^{2}(X,\G_{m,X})$ is in the image of $\alpha_{X}'$. \end{proposition} \begin{proof} Let $\ms{Y}$ be the $\G_{m,Y}$-gerbe corresponding to $\beta$. Set $\ms{X} := X \times_{Y} \ms{Y}$ and let $F : \ms{X} \to \ms{Y}$ be the induced morphism of algebraic stacks. If $\ms{X}$ is in the image of $\alpha_{X}'$, then there exists a $1$-twisted finite locally free $\mc{O}_{\ms{X}}$-module $\ms{E}$ of everywhere positive rank. The pushforward $F_{\ast}\ms{E}$ is a $1$-twisted, finite locally free $\mc{O}_{\ms{Y}}$-module of everywhere positive rank. Hence $\ms{Y}$ is in the image of $\alpha_{Y}'$. \par The other direction follows from commutativity of the diagram \labelcref{eqn-10}. \end{proof}

\begin{corollary} \label{11} Let $f : X \to Y$ be a finitely presented, finite, flat, surjective morphism of algebraic stacks. If $\alpha_{X}$ is an isomorphism, then $\alpha_{Y}$ is an isomorphism. \end{corollary}

\begin{corollary} \label{13} Let $X$ be a smooth separated generically tame Deligne-Mumford stack over a field $k$ with quasi-projective coarse moduli space. Then the Brauer map $\alpha_{X}$ is surjective. \end{corollary} \begin{proof} By Kresch-Vistoli \cite[2.1,2.2]{KV2004}, such $X$ has a finite flat surjection $Z \to X$ where $Z$ is a quasi-projective $k$-scheme. By Gabber's theorem (see \cite[1.1]{DEJONG-GABBER}), the Brauer map is surjective for $Z$. Thus the Brauer map is surjective for $X$ by \Cref{01}. \end{proof}

\begin{remark} \label{14} If $\on{char} k \ne 2$, the stack $\ms{M}_{1,1,k}$ is generically tame and so \Cref{13} implies surjectivity of the Brauer map $\alpha_{\ms{M}_{1,1,k}}$. For the case $\on{char} k = 2$, see \Cref{20170321-05}. \end{remark}

\section{Preliminary observations} \label{sec-04}

The purpose of this section is to prove \Cref{51} below. Let us start, however, with a few preliminary observations about the stack $\ms{M}_{1,1}$ and its Brauer group.

The stack $\ms{M}_{1,1,\Z}$ is a Deligne-Mumford stack smooth and separated over $\Z$ \cite[13.1.2]{OLSSON}; hence if $S$ is a regular Noetherian scheme then $\ms{M}_{1,1,S}$ is a regular Noetherian stack. For any locally Noetherian scheme $S$, the morphism \[ \pi : \ms{M}_{1,1,S} \to \A_{S}^{1} \] sending an elliptic curve to its $j$-invariant identifies $\A_{S}^{1}$ with the coarse moduli space of $\ms{M}_{1,1,S}$ \cite[4.4]{FULTON-OLSSON-PICARD}. \par In general, if $\ms{X}$ is a separated Deligne-Mumford stack and $\pi : \ms{X} \to X$ is its coarse moduli space, then $\pi$ is initial among maps from $\ms{X}$ to an algebraic space, so the map $X(\mc{G}) \to \ms{X}(\mc{G})$ is an isomorphism for any group scheme $\mc{G}$; moreover if $U \to X$ is an etale morphism, then $\pi_{U} : \ms{X} \times_{X} U \to U$ is a coarse moduli space. Applying these observations to $\mc{G} = \G_{a},\G_{m},\mu_{n}$ implies that the canonical maps $\mc{O}_{X} \to \pi_{\ast}\mc{O}_{\ms{X}}$, $\G_{m,X} \to \pi_{\ast}\G_{m,\ms{X}}$, $\mu_{n,X} \to \pi_{\ast}\mu_{n,\ms{X}}$ are isomorphisms; thus we will omit subscripts and denote $\mu_{n},\G_{m}$ for the corresponding sheaves on either $\ms{M}_{1,1,S}$ or $\A_{S}^{1}$. 

\begin{lemma} \label{20170321-05} Let $S$ be a quasi-compact scheme admitting an ample line bundle, and suppose that at least one prime $p$ is invertible in $S$. Then the Brauer map $\alpha_{\ms{M}_{1,1,S}} : \Br \ms{M}_{1,1,S} \to \Br' \ms{M}_{1,1,S}$ is an isomorphism. \end{lemma} \begin{proof} By \cite[4.7.2]{KATZ-MAZUR-AMOEC}, for $N \ge 3$ the moduli stack of full level $N$ structures is representable by an affine $\Z[\frac{1}{N}]$-scheme $Y(N)$. Set $Y(N)_{S} := Y(N) \times_{\Z[\frac{1}{N}]} S$; the projection $Y(N)_{S} \to S$ is an affine morphism, hence $Y(N)_{S}$ is quasi-compact and admits an ample line bundle, hence the Brauer map $\alpha_{Y(N)_{S}}$ is surjective by Gabber's theorem (see \cite{DEJONG-GABBER}), and, since the map $Y(N)_{S} \to \ms{M}_{1,1,S}$ is finite locally free, we have by \Cref{11} that $\alpha_{\ms{M}_{1,1,S}}$ is surjective. \end{proof}

\begin{lemma} \label{50} Let $U := \Spec \Z[t,(t(t-1728))^{-1}] \subset \A_{\Z}^{1}$ and let $\ms{M}_{1,1,\Z}^{\circ} := U \times_{\A_{\Z}^{1}} \ms{M}_{1,1,\Z}$. Then the restriction $\pi^{\circ} : \ms{M}_{1,1,\Z}^{\circ} \to U$ of $\pi$ to $U$ is a trivial $\Z/(2)$-gerbe, i.e. $\ms{M}_{1,1,\Z}^{\circ} \simeq \mathrm{B}(\Z/(2))_{U}$. \end{lemma} \begin{proof} Let $S$ be a scheme and let $E_{1},E_{2}$ be two elliptic curves over $S$. If $j(E_{1}) = j(E_{2}) \in \Gamma(S,\mc{O}_{S})$ and $j(E_{i}),j(E_{i})-1728$ are units of $\Gamma(S,\mc{O}_{S})$, then by \cite[5.3]{DELIGNE-COURBES-ELLIPTIQUES} one can find a finite \'etale cover $S' \to S$ such that there is an isomorphism $S' \times_{S} E_{1} \simeq S' \times_{S} E_{2}$ of elliptic curves over $S'$. For any connected scheme $S$ and an elliptic curve $E/S$ for which $j(E)$ and $j(E)-1728$ are invertible, we have $\Aut(E/S) \simeq \Z/(2)$ by \cite[(8.4.2)]{KATZ-MAZUR-AMOEC}. It suffices now to show that there is an elliptic curve $E_{U}$ over $U$ with $j$-invariant $t$. For this we may take the elliptic curve $E_{U}$ defined by the Weierstrass equation \[ \textstyle Y^{2}Z + XYZ = X^{3} - \frac{36}{t-1728} XZ^{2} - \frac{1}{t-1728} Z^{3} \] which satisfies $\Delta(E_{U}) = \frac{t^{2}}{(t-1728)^{3}}$ and $j(E_{U}) = t$ (see \cite[Proposition III.1.4(c)]{SILVERMAN}). \end{proof}

\begin{lemma} \label{49} Let $k$ be an algebraically closed field and let $U$ be a smooth curve over $k$. If $\Pic(U) = 0$, then $\Br' \mathrm{B}(\Z/(2))_{U} \simeq (\G_{m}(U))/(2)$. \end{lemma} \begin{proof} The cohomological descent spectral sequence associated to the cover $U \to \mathrm{B}(\Z/(2))_{U}$ is of the form \begin{align} \label{49-eqn-13} \on{E}_{2}^{p,q} = \H^{p}(\Z/(2) , \H_{\et}^{q}(U,\G_{m})) \implies \H_{\et}^{p+q}(\mathrm{B}(\Z/(2))_{U} , \G_{m}) \end{align} with differentials $\on{E}_{2}^{p,q} \to \on{E}_{2}^{p+2,q-1}$. We have by \cite[III.2.22 (d)]{MILNE-ETALE-COH-BOOK} that $\H_{\et}^{q}(U,\G_{m}) = 0$ for all $q \ge 2$. Moreover, we have $\H_{\et}^{1}(U,\G_{m}) = \Pic(U) = 0$ by assumption. Thus the only row of the $\on{E}_{2}$-page of \labelcref{49-eqn-13} containing nonzero entries is $q=0$, which gives an isomorphism \[ \H_{\et}^{2}(\mathrm{B}(\Z/(2))_{U} , \G_{m}) \simeq \H^{2}(\Z/(2) , \H_{\et}^{0}(U,\G_{m})) \simeq (\G_{m}(U))/(2) \] of abelian groups. \end{proof}

\begin{lemma} \label{51} Let $k$ be an algebraically closed field. If $\on{char} k \ne 2,3$, then $\Br' \ms{M}_{1,1,k}$ is a subgroup of $\Z/(2) \oplus \Z/(2)$. If $\on{char} k$ is $2$ or $3$, then $\Br' \ms{M}_{1,1,k}$ is a subgroup of $\Z/(2)$. \end{lemma} \begin{proof} We have that $\ms{M}_{1,1,k}$ is regular Noetherian and that $\ms{M}_{1,1,k}^{\circ} := \ms{M}_{1,1,\Z}^{\circ} \times_{\Z} k$ is a dense open substack; thus by \cite[2.5(iv)]{ANTIEAU-MEIER-TBGOTMSOEC} the map \[ \Br' \ms{M}_{1,1,k} \to \Br' \ms{M}_{1,1,k}^{\circ} \] induced by restriction is an injection. Here \Cref{50} implies $\Br' \ms{M}_{1,1,k}^{\circ} = \Br' \mathrm{B}(\Z/(2))_{U}$ for $U = \Spec k[t,(t(t-1728))^{-1}]$, and \Cref{49} implies $\Br' \mathrm{B}(\Z/(2))_{U}$ is $\Z/(2) \oplus \Z/(2)$ if $\on{char} k \ne 2,3$ and $\Z/(2)$ otherwise (here we use that $k^{\times} = (k^{\times})^{2}$ since $k$ is algebraically closed). \end{proof}

\section{The case $\on{char} k$ is not $2$} \label{sec-02}

Antieau and Meier \cite{ANTIEAU-MEIER-TBGOTMSOEC} compute the Brauer group $\Br \ms{M}_{1,1,S}$ for various base schemes $S$, including algebraically closed fields $k$ of odd characteristic \cite[11.2]{ANTIEAU-MEIER-TBGOTMSOEC} (the case $\on{char} k \ne 2$ in \Cref{48}). In this section we give a proof via a d\'evissage argument, using the fact that the coarse moduli space morphism $\pi : \ms{M} \to \A_{k}^{1}$ is a trivial $\Z/(2)$-gerbe away from $0,1728 \in \A_{k}^{1}$ (see \Cref{50}). Our proof is divided into two cases, depending on whether $\on{char} k = 3$ or $\on{char} k \ne 3$ (this will determine whether we puncture $\A_{k}^{1}$ at one or two points, respectively). We first fix notation and record some observations that apply to both cases.

\begin{pg} We abbreviate $\ms{M} := \ms{M}_{1,1,k}$. By \Cref{20170321-05}, the Brauer map $\alpha_{\ms{M}} : \Br \ms{M} \to \Br' \ms{M}$ is an isomorphism. By \Cref{51}, the main task is to show that the $2$-torsion in $\Br \ms{M}$ is $0$. \par For any integer $n \ge 1$, the \'etale Kummer sequence \[ 1 \to \mu_{2^{n}} \to \G_{m} \stackrel{\times 2^{n}}{\to} \G_{m} \to 1 \] gives an exact sequence \begin{align} \label{39-eqn-09} 0 \to (\Pic \ms{M}) / (2^{n}) \to \H^{2}(\ms{M},\mu_{2^{n}}) \to \H^{2}(\ms{M} , \G_{m})[2^{n}] \to 0 \end{align} of abelian groups. Since we have $\Pic \ms{M} \simeq \Z/(12)$ by \cite{FULTON-OLSSON-PICARD}, we wish to compute $\H^{2}(\ms{M} , \mu_{2^{n}})$. \par Set \[ U := \Spec k[t,(t(t-1728))^{-1}] = \A_{k}^{1} \setminus \{0,1728\} \] with inclusion $j : U \to \A_{k}^{1}$ and let $i : Z \to \A_{k}^{1}$ be the complement with reduced induced closed subscheme structure. (Thus, if $\on{char} k$ is $2$ or $3$ then $Z \simeq \Spec k$, otherwise $Z \simeq \Spec k \amalg \Spec k$.) Set \begin{align*} \ms{M}^{\circ} &:= U \times_{\A_{k}^{1}} \ms{M} \\ \ms{M}_{Z} &:= Z \times_{\A_{k}^{1}} \ms{M} \end{align*} with projections $\pi^{\circ} : \ms{M}^{\circ} \to U$ and $\pi_{Z} : \ms{M}_{Z} \to Z$. We have a commutative diagram \begin{equation} \label{39-eqn-13} \begin{tikzpicture}[>=angle 90, baseline=(current bounding box.center)] 
\matrix[matrix of math nodes,row sep=2em, column sep=2em, text height=1.5ex, text depth=0.25ex] { 
|[name=11]| \ms{M}^{\circ} & |[name=12]| \ms{M} & |[name=13]| \ms{M}_{Z} \\ 
|[name=21]| U & |[name=22]| \A_{k}^{1} & |[name=23]| Z \\
}; 
\draw[->,font=\scriptsize]
(11) edge (12) (21) edge node[below=-1pt] {$j$} (22) (11) edge node[left=-1pt] {$\pi^{\circ}$} (21) (12) edge node[right=-1pt] {$\pi$} (22) (23) edge node[below=-1pt] {$i$} (22) (13) edge (12) (13) edge node[right=-1pt] {$\pi_{Z}$} (23); \end{tikzpicture} \end{equation} with cartesian squares.

We have a distinguished triangle \begin{align} \label{39-eqn-10} j_{!}j^{\ast}\mathbf{R}\pi_{\ast}\mu_{2^{n}} \to \mathbf{R}\pi_{\ast}\mu_{2^{n}} \to i_{\ast}i^{\ast}\mathbf{R}\pi_{\ast}\mu_{2^{n}} \stackrel{+1}{\to} \end{align} in the derived category of bounded-below complexes of abelian sheaves on the \'etale site of $\A_{k}^{1}$, whose associated long exact sequence has the form \begin{equation} \label{39-eqn-07} \begin{tikzpicture}[>=angle 90, trim left,trim right=0pt, baseline=(current bounding box.center)] 
\matrix[matrix of math nodes,row sep=1em, column sep={8em,between origins}, text height=1.5ex, text depth=0.25ex] { 
|[name=11]| \H^{0}(\A_{k}^{1},j_{!}\mathbf{R}\pi^{\circ}_{\ast}\mu_{2^{n}}) & |[name=12]| \H^{0}(\ms{M},\mu_{2^{n}}) & |[name=13]| \H^{0}(Z,i^{\ast}\mathbf{R}\pi_{\ast}\mu_{2^{n}}) \\ 
|[name=21]| \H^{1}(\A_{k}^{1},j_{!}\mathbf{R}\pi^{\circ}_{\ast}\mu_{2^{n}}) & |[name=22]| \H^{1}(\ms{M},\mu_{2^{n}}) & |[name=23]| \H^{1}(Z,i^{\ast}\mathbf{R}\pi_{\ast}\mu_{2^{n}}) \\ 
|[name=31]| \H^{2}(\A_{k}^{1},j_{!}\mathbf{R}\pi^{\circ}_{\ast}\mu_{2^{n}}) & |[name=32]| \H^{2}(\ms{M},\mu_{2^{n}}) & |[name=33]| \H^{2}(Z,i^{\ast}\mathbf{R}\pi_{\ast}\mu_{2^{n}}) \\ 
}; 
\draw[->,font=\scriptsize]
(11) edge (12) (12) edge (13) (13) edge[gray,out=355,in=175] (21)
(21) edge (22) (22) edge (23) (23) edge[gray,out=355,in=175] (31)
(31) edge (32) (32) edge (33); \end{tikzpicture} \end{equation} since $j^{\ast}\mathbf{R}\pi_{\ast}\mu_{2^{n}} \simeq \mathbf{R}\pi^{\circ}_{\ast}\mu_{2^{n}}$ and \begin{align*} \H^{s}(\A_{k}^{1},\mathbf{R}\pi_{\ast}\mu_{2^{n}}) &\simeq \H^{s}(\ms{M},\mu_{2^{n}}) \\ \H^{s}(\A_{k}^{1},i_{\ast}i^{\ast}\mathbf{R}\pi_{\ast}\mu_{2^{n}}) &\simeq \H^{s}(Z,i^{\ast}\mathbf{R}\pi_{\ast}\mu_{2^{n}}) \end{align*} for all $s$. We will first compute the groups $\H^{s}(\A_{k}^{1},j_{!}j^{\ast}\mathbf{R}\pi_{\ast}\mu_{2^{n}})$ in the left column of \labelcref{39-eqn-07}. \end{pg} 

\begin{lemma} \label{43} Let $k$ be an algebraically closed field, let $x_{1},\dotsc,x_{r} \in \A_{k}^{1}$ be $r$ distinct $k$-points, set \[ Z := \Spec k(x_{1}) \amalg \dotsb \amalg \Spec k(x_{r}) \] and let $U = \A_{k}^{1} \setminus Z$ be the complement with inclusion $j : U \to \A_{k}^{1}$. For any positive integer $\ell$ invertible in $k$, we have \[ \H^{s}(\A_{k}^{1},j_{!}\mu_{\ell}) = \begin{cases} 0 &s \ne 1 \\ (\mu_{\ell}(k))^{\oplus (r-1)}&s = 1 \end{cases}. \] \end{lemma} \begin{proof} Let $i : Z \to \A_{k}^{1}$ be the inclusion. We have a distinguished triangle \[ j_{!}\mu_{\ell}|_{U} \to \mu_{\ell} \to i_{\ast}i^{\ast}\mu_{\ell} \stackrel{+1}{\to} \] in the derived category of bounded-below complexes of abelian sheaves on the big \'etale site of $\A_{k}^{1}$, which gives a long exact sequence \begin{center} \begin{tikzpicture}[>=angle 90, trim left,trim right=0pt] 
\matrix[matrix of math nodes,row sep=1.3em, column sep={7.5em,between origins}, text height=1.5ex, text depth=0.25ex] { 
|[name=11]| \H^{0}(\A_{k}^{1},j_{!}\mu_{\ell}|_{U}) & |[name=12]| \H^{0}(\A_{k}^{1},\mu_{\ell}) & |[name=13]| \H^{0}(Z,\mu_{\ell}) \\ 
|[name=21]| \H^{1}(\A_{k}^{1},j_{!}\mu_{\ell}|_{U}) & |[name=22]| \H^{1}(\A_{k}^{1},\mu_{\ell}) & |[name=23]| \H^{1}(Z,\mu_{\ell}) \\ 
|[name=31]| \H^{2}(\A_{k}^{1},j_{!}\mu_{\ell}|_{U}) & |[name=32]| \H^{2}(\A_{k}^{1},\mu_{\ell}) & |[name=33]| \H^{2}(Z,\mu_{\ell}) \\ 
|[name=41]| \H^{3}(\A_{k}^{1},j_{!}\mu_{\ell}|_{U}) & |[name=42]| \dotsb & |[name=43]|  \\ 
}; 
\draw[->,font=\scriptsize]
(11) edge (12) (12) edge (13) (13) edge[gray,out=355,in=175] (21)
(21) edge (22) (22) edge (23) (23) edge[gray,out=355,in=175] (31)
(31) edge (32) (32) edge (33) (33) edge[gray,out=355,in=175] (41)
(41) edge (42); \end{tikzpicture} \end{center} in cohomology. The map $\H^{0}(\A_{k}^{1},\mu_{\ell}) \to \H^{0}(Z,\mu_{\ell})$ is identified with the diagonal map $\mu_{\ell}(k) \to (\mu_{\ell}(k))^{\oplus r}$. Since $k$ is algebraically closed, the etale site of $Z$ is trivial, hence $\H^{s}(Z,\mu_{\ell}) = 0$ for $s \ge 1$. By \cite[Exp. 1, III, (3.6)]{SGA4.5} we have $\H^{s}(\A_{k}^{1},\mu_{\ell}) = 0$ for $s \ge 2$. We have $\G_{m}(\A_{k}^{1}) \simeq \G_{m}(k)$ and the multiplication-by-$\ell$ map $\times \ell : \G_{m}(k) \to \G_{m}(k)$ is surjective; thus $\H^{1}(\A_{k}^{1},\mu_{\ell}) = \H^{1}(\A_{k}^{1},\G_{m})[\ell] = (\Pic \A_{k}^{1})[\ell] = 0$ by the Kummer sequence. \end{proof}

\begin{lemma} \label{38} In the setup of \Cref{43}, let $n$ be any positive integer and let $\pi^{\circ} : \mathrm{B}(\Z/(n))_{U} \to U$ be the trivial $\Z/(n)$-gerbe over $U$. Then \[ \H^{s}(\A_{k}^{1} , j_{!}\mathbf{R}\pi^{\circ}_{\ast}\mu_{\ell}) = \begin{cases} 0 &\text{if }s = 0\;, \\ (\mu_{\ell}(k))^{\oplus (r-1)} &\text{if }s = 1\;, \\ (\mu_{\gcd(n,\ell)}(k))^{\oplus (r-1)} &\text{if }s = 2\;. \end{cases} \] \end{lemma} \begin{proof} We set \[ \mc{C} := j_{!}\mathbf{R}\pi^{\circ}_{\ast}\mu_{\ell} \] for convenience. We will compute the groups $\H^{s}(\A_{k}^{1},\mc{C})$ using the fact that the canonical truncations $\tau_{\le s}\mc{C}$ satisfy \begin{align} \label{39-eqn-12} \H^{s}(\A_{k}^{1},\tau_{\le t}\mc{C}) \simeq \H^{s}(\A_{k}^{1},\mc{C}) \end{align} for $s \le t$. For any $s \in \Z$, the distinguished triangle \begin{align} \label{39-eqn-06} \tau_{\le s-1}\mc{C} \to \tau_{\le s}\mc{C} \to (h^{s}\mc{C})[-s] \stackrel{+1}{\to} \end{align} gives a long exact sequence \begin{equation} \label{39-eqn-05} \begin{tikzpicture}[>=angle 90, trim left,trim right=0pt, baseline=(current bounding box.center)] 
\matrix[matrix of math nodes,row sep=1.3em, column sep=2em, text height=1.5ex, text depth=0.25ex] { 
|[name=11]| \H^{0}(\A_{k}^{1},\tau_{\le s-1}\mc{C}) & |[name=12]| \H^{0}(\A_{k}^{1},\tau_{\le s}\mc{C}) & |[name=13]| \H^{0-s}(\A_{k}^{1},j_{!}\mathbf{R}^{s}\pi^{\circ}_{\ast}\mu_{\ell}) \\ 
|[name=21]| \H^{1}(\A_{k}^{1},\tau_{\le s-1}\mc{C}) & |[name=22]| \H^{1}(\A_{k}^{1},\tau_{\le s}\mc{C}) & |[name=23]| \H^{1-s}(\A_{k}^{1},j_{!}\mathbf{R}^{s}\pi^{\circ}_{\ast}\mu_{\ell}) \\ 
|[name=31]| \H^{2}(\A_{k}^{1},\tau_{\le s-1}\mc{C}) & |[name=32]| \H^{2}(\A_{k}^{1},\tau_{\le s}\mc{C}) & |[name=33]| \H^{2-s}(\A_{k}^{1},j_{!}\mathbf{R}^{s}\pi^{\circ}_{\ast}\mu_{\ell}) \\ 
}; 
\path[->,font=\scriptsize]
(11) edge (12) (12) edge (13) (13.east) edge[gray,out=350,in=170,looseness=0.8] (21.west)
(21) edge (22) (22) edge (23) (23.east) edge[gray,out=350,in=170,looseness=0.8] (31.west)
(31) edge (32) (32) edge (33); \end{tikzpicture} \end{equation} where \[ h^{s}\mc{C} \simeq j_{!}\mathbf{R}^{s}\pi^{\circ}_{\ast}\mu_{\ell} \] since $j_{!}$ is exact. \par Since $\pi^{\circ} : \mathrm{B}(\Z/(n))_{U} \to U$ is a trivial $\Z/(n)$-gerbe, by \Cref{45} we have \begin{align} \label{39-eqn-11} \mathbf{R}^{s}\pi^{\circ}_{\ast}\mu_{\ell} \simeq \begin{cases} \mu_{\ell} & s = 0 \\ \mu_{\ell}[n] & s = 1,3,5,\dotsc \\ \mu_{\ell}/(n) & s = 2,4,6,\dotsc \end{cases} \end{align} where $\mu_{\ell}[n]$ and $\mu_{\ell}/(n)$ are defined by the exact sequence \[ 1 \to \mu_{\ell}[n] \to \mu_{\ell} \stackrel{\times n}{\to} \mu_{\ell} \to \mu_{\ell}/(n) \to 1 \] of abelian sheaves. Since $k$ is algebraically closed of characteristic prime to $\ell$, the sheaves $\mu_{\ell}[n]$ and $\mu_{\ell}/(n)$ are both isomorphic to $\mu_{\gcd(n,\ell)}$, but for us the difference is important for reasons of functoriality (as $\ell$ is allowed to vary). More precisely, if $\ell_{1}$ divides $\ell_{2}$, then the inclusion $\mu_{\ell_{1}} \to \mu_{\ell_{2}}$ induces an inclusion \[ \mu_{\ell_{1}}[n] \to \mu_{\ell_{2}}[n] \] whereas \begin{align} \label{38-eqn-01} \mu_{\ell_{1}}/(n) \to \mu_{\ell_{2}}/(n) \end{align} is not necessarily injective since an element $x \in \mu_{\ell_{1}}$ which is not an $n$th power of any $y_{1} \in \mu_{\ell_{1}}$ may be an $n$th power of some $y_{2} \in \mu_{\ell_{2}}$ (in particular, if $\ell_{2} = n \ell_{1}$, then \labelcref{38-eqn-01} is the zero morphism). \par We have \[ \tau_{\le 0}\mc{C} \simeq h^{0}\mc{C} \simeq j_{!}\mathbf{R}^{0}\pi^{\circ}_{\ast}\mu_{\ell} \simeq j_{!}\pi^{\circ}_{\ast}\mu_{\ell} \simeq j_{!}\mu_{\ell} \] since $\pi^{\circ}$ is a coarse moduli space morphism and $\mathbf{R}^{1}\pi^{\circ}_{\ast}\mu_{\ell} \simeq \mu_{\gcd(n,\ell)}$ by \labelcref{39-eqn-11}. Applying \Cref{43} to the case $s = 1$ in \labelcref{39-eqn-05} implies $\H^{0}(\A_{k}^{1},\tau_{\le 1}\mc{C}) = 0$ and gives isomorphisms $\H^{1}(\A_{k}^{1},j_{!}\mu_{\ell}) \simeq \H^{1}(\A_{k}^{1},\tau_{\le 1}\mc{C})$ and $\H^{2}(\A_{k}^{1},\tau_{\le 1}\mc{C}) \simeq \H^{1}(\A_{k}^{1},j_{!}\mu_{\gcd(n,\ell)})$. \par Since $\mathbf{R}^{2}\pi^{\circ}_{\ast}\mu_{\ell} \simeq \mu_{\gcd(n,\ell)}$ by \labelcref{39-eqn-11} and $\H^{s}(\A_{k}^{1},j_{!}\mu_{\gcd(n,\ell)}) = 0$ for $s = -2,-1,0$, the case $s = 2$ in \labelcref{39-eqn-05} gives isomorphisms $\H^{s}(\A_{k}^{1},\tau_{\le 1}\mc{C}) \simeq \H^{s}(\A_{k}^{1},\tau_{\le 2}\mc{C})$ for $s = 0,1,2$, which implies the desired result. \end{proof}

\begin{pg}[Proof of {\Cref{48}} for $\on{char} k = 3$] If $\on{char} k = 3$, then $Z$ consists of one point, so taking $r = 1$ in \Cref{38} implies \begin{align} \label{39-eqn-14} \H^{s}(\A_{k}^{1},j_{!}\mathbf{R}\pi^{\circ}_{\ast}\mu_{2^{n}}) = 0 \end{align} for $s = 0,1,2$. Therefore, to compute $\H^{2}(\ms{M},\mu_{2^{n}})$, it now remains to compute $\H^{2}(Z,i^{\ast}\mathbf{R}\pi_{\ast}\mu_{2^{n}})$ in \labelcref{39-eqn-07}. The stabilizer of any object of $\ms{M}$ of lying over $i : Z \to \A_{k}^{1}$ is the automorphism group of an elliptic curve with $j$-invariant $0$, which is the semidirect product $\Gamma = \Z/(3) \rtimes \Z/(4)$ since $k$ has characteristic $3$. The underlying reduced stack $(\ms{M}_{Z})_{\on{red}}$ is the residual gerbe associated to the unique point of $|\ms{M}_{Z}|$ and is isomorphic to the classifying stack $\mathrm{B}\Gamma_{k}$. We have natural isomorphisms \[ \H^{2}(Z,i^{\ast}\mathbf{R}\pi_{\ast}\mu_{2^{n}}) \simeq i^{\ast}\mathbf{R}^{2}\pi_{\ast}\mu_{2^{n}} \stackrel{1}{\simeq} \H^{2}(\ms{M}_{Z} , \mu_{2^{n}}) \stackrel{2}{\simeq} \H^{2}(\mathrm{B}\Gamma_{k} , \mu_{2^{n}}) \stackrel{3}{\simeq} \H^{2}(\Gamma , \mu_{2^{n}}(k)) \] where isomorphism 1 follows from proper base change \cite[1.3]{OLSSON-ONPROPERCOVERINGSOFARTINSTACKS}, isomorphism 2 is by invariance of \'etale site for nilpotent thickenings and the fact that $2^{n}$ is invertible on $\ms{M}_{Z}$, and isomorphism 3 is by the cohomological descent spectral sequence for the covering $\Spec k \to \mathrm{B}\Gamma_{k}$ (and the fact that $\H^{i}(\Spec k , \mu_{2^{n}}) = 0$ for $i > 0$ since $k$ is algebraically closed). The Hochschild-Serre spectral sequence for the exact sequence \[ 1 \to \Z/(3) \to \Gamma \to \Z/(4) \to 1 \] gives an isomorphism \[ \H^{2}(\Gamma , \mu_{2^{n}}(k)) \simeq \H^{2}(\Z/(4),\mu_{2^{n}}(k)) \simeq \mu_{2^{n}}(k)/(4) \] where $\H^{i}(\Z/(3) , \mu_{2^{n}}(k)) = 0$ for $i > 0$ since $3$ is coprime to the order of $\mu_{2^{n}}(k)$. Since the first term in the last row of the diagram \labelcref{39-eqn-07} is zero by \labelcref{39-eqn-14}, the above observations imply that we have natural inclusions \[ \H^{2}(\ms{M},\mu_{2^{n}}) \to \mu_{2^{n}}(k)/(4) \] compatible with the inclusions $\mu_{2^{n}} \subset \mu_{2^{n+1}}$ for all $n$. The inclusion  $\mu_{2^{n}} \subset \mu_{2^{n+2}}$ induces the zero map $\mu_{2^{n}}(k)/(4) \to \mu_{2^{n+2}}(k)/(4)$, so $\H^{2}(\ms{M},\mu_{2^{n}}) \to \H^{2}(\ms{M},\mu_{2^{n+2}})$ is the zero map as well, hence \[ \textstyle \varinjlim_{n \in \N} \H^{2}(\ms{M},\mu_{2^{n}}) = 0 \] which by \labelcref{39-eqn-09} gives $\H^{2}(\ms{M},\G_{m})[2^{n}] = 0$ for all $n$. \end{pg}

\begin{pg}[Proof of {\Cref{48}}, for $\on{char} k \ne 2,3$] We describe the terms in \labelcref{39-eqn-07}. For the right column, we have \[ \H^{s}(Z,i^{\ast}\mathbf{R}\pi_{\ast}\mu_{2^{n}}) \simeq \H^{s}(\Z/(4) , \mu_{2^{n}}(k)) \oplus \H^{s}(\Z/(6) , \mu_{2^{n}}(k)) \] by \cite[A.0.7]{ABRAMOVICH-CORTI-VISTOLI-TWISTED-BUNDLES-AND-ADMISSIBLE-COVERS}. For the middle column, we have \[ \H^{0}(\ms{M},\mu_{2^{n}}) \simeq \H^{0}(\A_{k}^{1},\mu_{2^{n}}) \simeq \mu_{2^{n}}(k) \] since $\A_{k}^{1}$ is the coarse moduli space of $\ms{M}$, and we have \[ \H^{1}(\ms{M},\mu_{2^{n}}) \stackrel{1}{\simeq} \H^{1}(\ms{M},\G_{m})[2^{n}] \stackrel{2}{\simeq} (\Z/(12))[2^{n}] \stackrel{3}{\simeq} \Z/(4) \] where isomorphism 1 follows since $k^{\times} = (k^{\times})^{2^{n}}$, isomorphism 2 is by \cite{MUMFORD-PGOMP}, and isomorphism 3 holds for $n \gg 0$. For the left column, we have \[ \H^{s}(\tau_{\le 1}j_{!}\mathbf{R}\pi^{\circ}_{\ast}\mu_{2^{n}}) = \begin{cases} 0 & s=0 \\ \mu_{2^{n}} & s=1 \\ \mu_{2} & s=2 \end{cases} \] by \Cref{38}. \par To summarize, \labelcref{39-eqn-07} simplifies to \begin{equation} \label{39-eqn-08} \begin{tikzpicture}[>=angle 90, trim left,trim right=0pt, baseline=(current bounding box.center)] 
\matrix[matrix of math nodes,row sep=1.3em, column sep={8.5em,between origins}, text height=1.5ex, text depth=0.25ex] { 
|[name=11]| 0 & |[name=12]| \mu_{2^{n}} & |[name=13]| \mu_{2^{n}} \oplus \mu_{2^{n}} \\ 
|[name=21]| \mu_{2^{n}} & |[name=22]| \Z/(4) & |[name=23]| \mu_{4} \oplus \mu_{2} \\ 
|[name=31]| \mu_{2} & |[name=32]| \H^{2}(\ms{M},\mu_{2^{n}}) & |[name=33]| \mu_{2^{n}}/(4) \oplus \mu_{2^{n}}/(6) \\ 
}; 
\draw[->,font=\scriptsize]
(11) edge (12) (12) edge (13) (13) edge[gray,out=355,in=175] (21)
(21) edge (22) (22) edge (23) (23) edge[gray,out=355,in=175] (31)
(31) edge (32) (32) edge (33); \end{tikzpicture} \end{equation} for $n \gg 0$, and counting the number of elements in each group in \labelcref{39-eqn-08} implies that the last morphism \[ \H^{2}(\ms{M},\mu_{2^{n}}) \to \mu_{2^{n}}/(4) \oplus \mu_{2^{n}}/(6) \] is injective. Furthermore, the inclusion \[ \mu_{2^{n}} \subset \mu_{2^{n+2}} \] induces the zero map \[ \mu_{2^{n}}/(4) \oplus \mu_{2^{n}}/(6) \to \mu_{2^{n+2}}/(4) \oplus \mu_{2^{n+2}}/(6) \] so the map $\H^{2}(\ms{M},\mu_{2^{n}}) \to \H^{2}(\ms{M},\mu_{2^{n+2}})$ is the zero map as well, hence \[ \textstyle \varinjlim_{n \in \N} \H^{2}(\ms{M},\mu_{2^{n}}) = 0 \] which by \labelcref{39-eqn-09} gives $\H^{2}(\ms{M},\G_{m})[2^{n}] = 0$ for all $n$. \end{pg}

\section{The case $\on{char} k$ is $2$} \label{sec-03}

In this section we prove \Cref{48} (in case $\on{char} k = 2$) and \Cref{20170322-02}. For convenience, we denote $\GL_{n,p} := \GL_{n}(\Z/(p))$ and $\SL_{n,p} := \SL_{n}(\Z/(p))$. We denote by $\mathsf{e}$ the identity element of $\GL_{n,p}$. 

\begin{pg}[Hesse presentation of $\ms{M}_{1,1,k}$] \label{04-01} By \cite[6.2]{FULTON-OLSSON-PICARD} (and explained in more detail in \Cref{25}), there is a left action of $\GL_{2,3}$ on the $\Z[\frac{1}{3}]$-algebra \[ \textstyle A_{\mr{H}} := \Z[\frac{1}{3},\mu,\omega,\frac{1}{\mu^{3}-1}]/(\omega^{2}+\omega+1) \] sending \begin{equation} \label{04-eqn-03} \begin{aligned} \begin{bmatrix} 1 & 0 \\ 0 & -1 \end{bmatrix} \ast (\mu,\omega) &= (\mu,\omega^{2}) \\ \begin{bmatrix} 1 & 0 \\ -1 & 1 \end{bmatrix} \ast (\mu,\omega) &= (\omega\mu,\omega) \\ \begin{bmatrix} 0 & -1 \\ 1 & 0 \end{bmatrix} \ast (\mu,\omega) &= \textstyle (\frac{\mu+2}{\mu-1} , \omega) \end{aligned} \end{equation} for which the corresponding right action of $\GL_{2,3}$ on the $\Z[\frac{1}{3}]$-scheme \[ S_{\mr{H}} := \Spec A_{\mr{H}} \] gives a presentation \begin{align} \label{04-eqn-01} \ms{M}_{1,1,\Z[\frac{1}{3}]} \simeq [S_{\mr{H}}/\GL_{2,3}] \end{align} of $\ms{M}_{1,1,\Z[\frac{1}{3}]}$ as a global quotient stack. The morphism \begin{align} \label{04-eqn-09} S_{\mr{H}} \to \ms{M}_{1,1,\Z[\frac{1}{3}]} \end{align} is given by the elliptic curve \[ X^{3} + Y^{3} + Z^{3} = 3\mu XYZ \] over $S_{\mr{H}}$. \end{pg}

\begin{pg}[Cohomological descent] \label{04-02} Let $k$ be an algebraically closed field of characteristic $2$. The Brauer map $\alpha_{\ms{M}_{1,1,k}} : \Br \ms{M}_{1,1,k} \to \Br' \ms{M}_{1,1,k}$ is an isomorphism by \Cref{20170321-05}. By \Cref{51}, there is only $2$-torsion in $\Br \ms{M}_{1,1,k}$. By Grothendieck's fppf-\'etale comparison theorem for smooth commutative group schemes \cite[(11.7)]{Gro68c}, it suffices to compute the $2$-torsion in $\H_{\on{fppf}}^{2}(\ms{M}_{1,1,k} , \G_{m})$. Since $\Spec k$ is a reduced scheme, we have \[ \H_{\on{fppf}}^{1}(\ms{M}_{1,1,k} , \G_{m}) = \Pic(\ms{M}_{1,1,k}) = \Z/(12) \] by \cite[1.1]{FULTON-OLSSON-PICARD}. Thus, for any integer $n$, the fppf Kummer sequence \begin{align} \label{04-eqn-20} 1 \to \mu_{2} \to \G_{m} \stackrel{\times 2}{\to} \G_{m} \to 1 \end{align} gives an exact sequence \begin{align} \label{04-eqn-02} 1 \to \Z/(2) \stackrel{\partial}{\to} \H_{\on{fppf}}^{2}(\ms{M}_{1,1,k} , \mu_{2}) \to \H_{\on{fppf}}^{2}(\ms{M}_{1,1,k} , \G_{m})[2] \to 1 \end{align} of abelian groups. It remains to compute the middle term $\H_{\on{fppf}}^{2}(\ms{M}_{1,1,k} , \mu_{2})$. \par The cohomological descent spectral sequence associated to the cover \labelcref{04-eqn-09} is of the form \begin{align} \label{04-eqn-04} \on{E}_{2}^{p,q} = \H^{p}(\GL_{2,3} , \H_{\on{fppf}}^{q}(S_{\mr{H},k},\mu_{2})) \implies \H_{\on{fppf}}^{p+q}(\ms{M}_{1,1,k} , \mu_{2}) \end{align} with differentials $\on{E}_{2}^{p,q} \to \on{E}_{2}^{p+2,q-1}$. \par Let \[ \xi \in k \] be a fixed primitive $3$rd root of unity. By the Chinese Remainder Theorem, there is a $k$-algebra isomorphism \begin{align} \label{04-eqn-05} \textstyle A_{\mr{H},k} = k[\mu,\omega,\frac{1}{\mu^{3}-1}]/(\omega^{2}+\omega+1) \to k[\nu_{1},\frac{1}{\nu_{1}^{3}-1}] \times k[\nu_{2},\frac{1}{\nu_{2}^{3}-1}] \end{align} sending $\mu \mapsto (\nu_{1},\nu_{2})$ and $\omega \mapsto (\xi,\xi^{2})$. Since $S_{\mr{H},k}$ is a smooth curve over an algebraically closed field, we have by \cite[III.2.22 (d)]{MILNE-ETALE-COH-BOOK} that $\H_{\et}^{q}(S_{\mr{H},k},\G_{m}) = 0$ for all $q \ge 2$; since $S_{\mr{H},k}$ is a disjoint union of two copies of a distinguished affine open subset of $\A_{k}^{1}$, we have $\H_{\et}^{1}(S_{\mr{H},k},\G_{m}) = \Pic(S_{\mr{H},k}) = 0$. By \cite[(11.7)]{Gro68c} we have $\H_{\on{fppf}}^{q}(S_{\mr{H},k},\G_{m}) = \H_{\et}^{q}(S_{\mr{H},k},\G_{m})$ for all $q \ge 0$; thus the fppf Kummer sequence implies $\H_{\on{fppf}}^{q}(S_{\mr{H},k},\mu_{2}) = 0$ for all $q \ge 2$. Furthermore, we have $\H_{\on{fppf}}^{0}(S_{\mr{H},k},\mu_{2}) = 0$ since $S_{\mr{H},k}$ is the product of two integral domains of characteristic $2$. Thus the only nonzero terms on the $\on{E}_{2}$-page of \labelcref{04-eqn-04} occur on the $q=1$ row, so we have an isomorphism \begin{align} \label{04-eqn-12} \H_{\on{fppf}}^{p+1}(\ms{M}_{1,1,k} , \mu_{2}) \simeq \H^{p}(\GL_{2,3} , \H_{\on{fppf}}^{1}(S_{\mr{H},k},\mu_{2})) \end{align} for all $p \ge 0$. We are interested in the case $p = 1$. \end{pg}

\begin{pg}[Description of the $\GL_{2,3}$-action on $\H_{\on{fppf}}^{1}(S_{\mr{H},k},\mu_{2})$] \label{04-04} We describe the abelian group \[ M := \H_{\on{fppf}}^{1}(S_{\mr{H},k},\mu_{2}) \] and the left $\GL_{2,3}$-module structure it inherits from \labelcref{04-eqn-03}. Since $k[\mu,(\mu^{3}-1)^{-1}]$ is a principal localization of the polynomial ring $k[\mu]$ by a polynomial $\mu^{3}-1 = (\mu-1)(\mu-\xi)(\mu-\xi^{2})$ splitting into three distinct irreducible factors, we have an isomorphism \begin{align} \label{04-eqn-06} \textstyle (k[\mu,\frac{1}{\mu^{3}-1}])^{\times} \simeq k^{\times} \cdot (\mu-1)^{\Z} \cdot (\mu-\xi)^{\Z} \cdot (\mu-\xi^{2})^{\Z} \end{align} of abelian groups. Thus \labelcref{04-eqn-05} and the Kummer sequence \labelcref{04-eqn-20} gives an isomorphism \begin{align} \label{04-eqn-07} M \simeq (\Z/(2))^{\oplus 6} \end{align} of abelian groups, with generators given by the classes of $\nu_{i} - \xi^{j}$ for $i=1,2$ and $j = 0,1,2$. \par The isomorphism \labelcref{04-eqn-05} is given by the map \begin{align} \label{04-02-eqn-22} s_{1}(\mu)\omega + s_{0}(\mu) \mapsto \big( s_{1}(\nu_{1})\xi+s_{0}(\nu_{1}) , s_{1}(\nu_{2})\xi^{2}+s_{0}(\nu_{2}) \big) \end{align} for $s_{0},s_{1} \in k[\mu,\frac{1}{\mu^{3}-1}]$. The inverse of \labelcref{04-eqn-05} is given by the map \begin{align} \label{04-02-eqn-21} (f_{1}(\nu_{1}),f_{2}(\nu_{2})) \mapsto f_{1}(\mu)\left(\frac{\omega}{\xi-\xi^{2}} + \frac{\xi}{\xi-1}\right) + f_{2}(\mu)\left(\frac{-\omega}{\xi-\xi^{2}} + \frac{-1}{\xi-1}\right) \end{align} where $f_{i}(\nu_{i}) \in k[\nu_{i},\frac{1}{\nu_{i}^{3}-1}]$. (Note that, if we set $\mr{A}_{1}(t) := \frac{t}{\xi-\xi^{2}} + \frac{\xi}{\xi-1}$ and $\mr{A}_{2}(t) := \frac{-t}{\xi-\xi^{2}} + \frac{-1}{\xi-1}$, then $\mr{A}_{i}(\xi^{j})$ is the Kronecker delta function.) \par A computation with \labelcref{04-eqn-03}, \labelcref{04-02-eqn-22}, \labelcref{04-02-eqn-21} shows that the action of $\GL_{2,3}$ on the right hand side of \labelcref{04-eqn-05} is given by \begin{equation} \label{04-02-eqn-20} \begin{aligned} \begin{bmatrix} 1 & 0 \\ 0 & -1 \end{bmatrix} \ast (f_{1}(\nu_{1}),f_{2}(\nu_{2})) &= (f_{2}(\nu_{1}) , f_{1}(\nu_{2})) \\ \begin{bmatrix} 1 & 0 \\ -1 & 1 \end{bmatrix} \ast (f_{1}(\nu_{1}),f_{2}(\nu_{2})) &= (f_{1}(\xi\nu_{1}) , f_{2}(\xi^{2}\nu_{2})) \\ \begin{bmatrix} 0 & -1 \\ 1 & 0 \end{bmatrix} \ast (f_{1}(\nu_{1}),f_{2}(\nu_{2})) &= \textstyle (f_{1}(\frac{\nu_{1}+2}{\nu_{1}-1}),f_{2}(\frac{\nu_{2}+2}{\nu_{2}-1})) \end{aligned} \end{equation} for $f_{i}(\nu_{i}) \in k[\nu_{i},\frac{1}{\nu_{i}^{3}-1}]$. A computation with \labelcref{04-02-eqn-20} (and using that $\on{char} k =2$) shows that the action of $\GL_{2,3}$ on \labelcref{04-eqn-07} is given by \labelcref{04-eqn-08}, where every element is considered up to multiplication by $k^{\times}$. \begin{equation} \label{04-eqn-08} {\small \begin{array}{r|c|c|c|c|c|c} & \nu_{1}-1 & \nu_{1}-\xi & \nu_{1}-\xi^{2} & \nu_{2}-1 & \nu_{2}-\xi & \nu_{2}-\xi^{2} \\[1em] \mathsf{M}_{1} := \begin{bmatrix} 1 & 0 \\ 0 & -1 \end{bmatrix} & \nu_{2}-1 & \nu_{2}-\xi & \nu_{2}-\xi^{2} & \nu_{1}-1 &\nu_{1}-\xi & \nu_{1}-\xi^{2} \\[1em] \mathsf{M}_{2} := \begin{bmatrix} 1 & 0 \\ -1 & 1 \end{bmatrix} & \nu_{1}-\xi^{2} & \nu_{1}-1 & \nu_{1}-\xi & \nu_{2}-\xi & \nu_{2}-\xi^{2} & \nu_{2}-1 \\[1em] \mathsf{i} := \begin{bmatrix} 0 & -1 \\ 1 & 0 \end{bmatrix} & \displaystyle \frac{1}{\nu_{1}-1} & \displaystyle \frac{\nu_{1}-\xi^{2}}{\nu_{1}-1} & \displaystyle \frac{\nu_{1}-\xi}{\nu_{1}-1} & \displaystyle \frac{1}{\nu_{2}-1} & \displaystyle \frac{\nu_{2}-\xi^{2}}{\nu_{2}-1} & \displaystyle \frac{\nu_{2}-\xi}{\nu_{2}-1} \end{array} } \end{equation} \end{pg}

\begin{pg} \label{04-03} We compute $\H^{1}(\GL_{2,3},M)$. (In \Cref{sec-07} we provide \textsc{Magma} code that can be used to verify this computation.) We have a filtration of groups \begin{align} \label{04-eqn-14} \mathrm{Q}_{8} \trianglelefteq \SL_{2,3} \trianglelefteq \GL_{2,3} \end{align} where each is a normal subgroup of the next. Here $\mathrm{Q}_{8}$ denotes the quaternion group \[ \mathrm{Q}_{8} = \{\pm \mathsf{e} , \pm \mathsf{i} , \pm \mathsf{j} , \pm \mathsf{k} \;:\; \mathsf{i}\mathsf{j}\mathsf{k} = \mathsf{i}^{2} = \mathsf{j}^{2} = \mathsf{k}^{2} = -\mathsf{e}\} \] and is identified with the subgroup of $\GL_{2,3}$ as follows: \[ \mathsf{i} = \begin{bmatrix} 0 & -1 \\ 1 & 0 \end{bmatrix} \qquad \mathsf{j} = \begin{bmatrix} -1 & -1 \\ -1 & 1 \end{bmatrix} \qquad \mathsf{k} = \begin{bmatrix} 1 & -1 \\ -1 & -1 \end{bmatrix} \] The quotient $\GL_{2,3}/\SL_{2,3}$ is cyclic of order $2$ and is generated by $\mathsf{M}_{1}$ in \labelcref{04-eqn-08}. The quotient $\SL_{2,3}/\mathrm{Q}_{8}$ is cyclic of order $3$ and is generated by $\mathsf{M}_{2}$ in \labelcref{04-eqn-08}. For $i=1,2$, let $\langle \mathsf{M}_{i} \rangle$ denote the subgroup of $\GL_{2,3}$ generated by $\mathsf{M}_{i}$. We note that $\SL_{2,3}$ is generated by $\mathsf{i}$ and $\mathsf{M}_{2}$. \par Let \[ F : (\Z[\GL_{2,3}]\on{-Mod}) \to (\Z[\SL_{2,3}]\on{-Mod}) \] be the forgetful functor. An inspection of \labelcref{04-eqn-08} implies that $F(M)$ is the direct sum $N_{1} \oplus N_{2}$ where $N_{i}$ is the $\SL_{2,3}$-submodule of $F(M)$ generated by the classes of $\nu_{i}-1,\nu_{i}-\xi,\nu_{i}-\xi^{2}$, and moreover $\mathsf{M}_{1}$ switches the summands $N_{1}$ and $N_{2}$. Under the adjunction \[ \Hom_{\SL_{2,3}}(F(M),N_{1}) \simeq \Hom_{\GL_{2,3}}(M,\on{Ind}_{\SL_{2,3}}^{\GL_{2,3}}(N_{1})) \] the projection map $F(M) \simeq N_{1} \oplus N_{2} \to N_{1}$ onto the first factor corresponds to a morphism \begin{align} \label{04-eqn-21} M \to \on{Ind}_{\SL_{2,3}}^{\GL_{2,3}}(N_{1}) \end{align} of $\GL_{2,3}$-modules. Given $m \in M$, write $m = n_{1} + n_{2}$ for $n_{i} \in N_{i}$; then the image of $m$ under \labelcref{04-eqn-21} is the function $\varphi_{m} \in \Hom_{\Z[\SL_{2,3}]}(\Z[\GL_{2,3}] , N_{1})$ such that $\varphi_{m}([\mathsf{e}]) = n_{1}$ and $\varphi_{m}([\mathsf{M}_{1}]) = \mathsf{M}_{1} \cdot n_{2}$; thus \labelcref{04-eqn-21} is an isomorphism. \par A computation using \labelcref{04-eqn-08} and the identities \begin{align} \label{04-eqn-16} \begin{aligned} \mathsf{k} &= \mathsf{M}_{2}^{-1} \cdot \mathsf{i} \cdot \mathsf{M}_{2} \\ \mathsf{i} &= \mathsf{M}_{2}^{-1} \cdot \mathsf{j} \cdot \mathsf{M}_{2} \\ \mathsf{j} &= \mathsf{M}_{2}^{-1} \cdot \mathsf{k} \cdot \mathsf{M}_{2} \end{aligned} \end{align} shows that the action of an element $\mathsf{g} \in \SL_{2,3}$ on $N_{1}$ is by left multiplication by the matrix $T_{\mathsf{g}}$ as in \labelcref{04-eqn-15}, with elements of $N_{1}$ being viewed as vertical vectors. We note $T_{-\mathsf{e}} = T_{\mathsf{i}}^{2} = T_{\mathsf{j}}^{2} = T_{\mathsf{k}}^{2} = \id_{N_{1}}$, i.e. $-\mathsf{e}$ acts trivially on $N_{1}$.\begin{equation} \label{04-eqn-15} \begin{array}{c|cccc} \mathsf{g} & \mathsf{M}_{2} & \mathsf{i} & \mathsf{j} & \mathsf{k} \\[0.5em] T_{\mathsf{g}} & \begin{bmatrix} 0 & 1 & 0 \\ 0 & 0 & 1 \\ 1 & 0 & 0 \end{bmatrix} & \begin{bmatrix} 1 & 1 & 1 \\ 0 & 0 & 1 \\ 0 & 1 & 0 \end{bmatrix} & \begin{bmatrix} 0 & 1 & 0 \\ 1 & 0 & 0 \\ 1 & 1 & 1 \end{bmatrix} & \begin{bmatrix} 0 & 0 & 1 \\ 1 & 1 & 1 \\ 1 & 0 & 0 \end{bmatrix} \end{array} \end{equation} Since $M$ is an induced module, the restriction map \begin{align} \label{04-eqn-13} \H^{1}(\GL_{2,3},M) \to \H^{1}(\SL_{2,3},N_{1}) \end{align} is an isomorphism so we reduce to computing $\H^{1}(\SL_{2,3},N_{1})$. \par The Hochschild-Serre spectral sequence for the inclusion $\mathrm{Q}_{8} \trianglelefteq \SL_{2,3}$ degenerates on the $\mathrm{E}_{2}$ page since the order of the quotient group $\langle \mathsf{M}_{2} \rangle$ is coprime to the order of $N_{1}$. In particular the restriction map \begin{align} \label{04-eqn-19} \H^{1}(\SL_{2,3} , N_{1}) \to \H^{0}(\langle \mathsf{M}_{2} \rangle , \H^{1}(\mathrm{Q}_{8},N_{1})) \end{align} is an isomorphism. \par Let $\mr{C}^{i}(\mathrm{Q}_{8},N_{1}) := \mr{Fun}((\mathrm{Q}_{8})^{i} , N_{1})$ denote the group of inhomogeneous $i$-cochains. By \Cref{06}, the group $\SL_{2,3}$ has a natural left action on $\mr{C}^{i}(\mathrm{Q}_{8},N_{1})$ (by entrywise conjugation on the source $(\mathrm{Q}_{8})^{i}$ and by its usual action on $N_{1}$) such that the differentials in the inhomogeneous cochain complex \[ \mr{C}^{0}(\mathrm{Q}_{8},N_{1}) \stackrel{d_{0}}{\to} \mr{C}^{1}(\mathrm{Q}_{8},N_{1}) \stackrel{d_{1}}{\to} \mr{C}^{2}(\mathrm{Q}_{8},N_{1}) \to \dotsb \] are $\SL_{2,3}$-linear. Since the order of the subgroup $\langle \mathsf{M}_{2} \rangle$ is coprime to the orders of $\mr{C}^{i}(\mathrm{Q}_{8},N_{1})$, we have that $\H^{0}(\langle \mathsf{M}_{2} \rangle , \H^{1}(\mathrm{Q}_{8},N_{1})) \simeq (\H^{1}(\mathrm{Q}_{8},N_{1}))^{\mathsf{M}_{2}}$ is isomorphic to the middle cohomology of the sequence \[ (\mr{C}^{0}(\mathrm{Q}_{8},N_{1}))^{\mathsf{M}_{2}} \stackrel{(d_{0})^{\mathsf{M}_{2}}}{\to} (\mr{C}^{1}(\mathrm{Q}_{8},N_{1}))^{\mathsf{M}_{2}} \stackrel{(d_{1})^{\mathsf{M}_{2}}}{\to} (\mr{C}^{2}(\mathrm{Q}_{8},N_{1}))^{\mathsf{M}_{2}}  \] i.e. cohomology commutes with taking $\mathsf{M}_{2}$-invariants. \par We now describe $\ker ((d_{1})^{\mathsf{M}_{2}})$ and $\im ((d_{0})^{\mathsf{M}_{2}})$. \par An element $f \in (\mr{C}^{1}(\mathrm{Q}_{8},N_{1}))^{\mathsf{M}_{2}}$ is a function $f : \mathrm{Q}_{8} \to N_{1}$ satisfying \begin{align} \label{04-eqn-17} f(\mathsf{g}) = \mathsf{M}_{2} \cdot f(\mathsf{M}_{2}^{-1} \mathsf{g} \mathsf{M}_{2}) \end{align} for all $\mathsf{g} \in \mathrm{Q}_{8}$. We have that $f \in \ker d_{1}$ if \begin{align} \label{04-eqn-18} f(\mathsf{g}_{1} \cdot \mathsf{g}_{2}) = \mathsf{g}_{1} \cdot f(\mathsf{g}_{2}) + f(\mathsf{g}_{1}) \end{align} for all $\mathsf{g}_{1},\mathsf{g}_{2} \in \mathrm{Q}_{8}$. \par Suppose $f \in \ker ((d_{1})^{\mathsf{M}_{2}}) = (\ker d_{1}) \cap (\mr{C}^{1}(\mathrm{Q}_{8},N_{1}))^{\mathsf{M}_{2}}$; taking $(\mathsf{g}_{1},\mathsf{g}_{2}) = (\mathsf{e},\mathsf{e})$ in \labelcref{04-eqn-18} implies $f(\mathsf{e}) = 0$; taking $\mathsf{g} = -\mathsf{e}$ in \labelcref{04-eqn-17} implies that \[ f(-\mathsf{e}) = (s,s,s) \] for some $s \in \Z/(2)$; taking $(\mathsf{g}_{1},\mathsf{g}_{2}) = (-\mathsf{e},-\mathsf{e})$ in \labelcref{04-eqn-18} and using the fact that $-\mathsf{e}$ acts trivially on $N_{1}$ implies that $2f(-\mathsf{e}) = 0$, which imposes no condition on $s$. We note that \[ \mathsf{g} \cdot f(-\mathsf{e}) = f(-\mathsf{e}) \] for any $\mathsf{g} \in \SL_{2,3}$. \par Setting $\mathsf{g} = \mathsf{i},\mathsf{j},\mathsf{k}$ in \labelcref{04-eqn-17} and using \labelcref{04-eqn-16} gives \begin{align} \label{04-eqn-22} \begin{aligned} f(\mathsf{i}) &= \mathsf{M}_{2} \cdot f(\mathsf{k}) \\ f(\mathsf{j}) &= \mathsf{M}_{2} \cdot f(\mathsf{i}) \\ f(\mathsf{k}) &= \mathsf{M}_{2} \cdot f(\mathsf{j}) \end{aligned} \end{align} respectively; thus we have \begin{align*} f(\mathsf{i}) &= (s_{1},s_{2},s_{3}) \\ f(\mathsf{j}) &= (s_{2},s_{3},s_{1}) \\ f(\mathsf{k}) &= (s_{3},s_{1},s_{2}) \end{align*} for some $s_{1},s_{2},s_{3} \in \Z/(2)$. \par Setting either $\mathsf{g}_{1} = -\mathsf{e}$ or $\mathsf{g}_{2} = -\mathsf{e}$ in \labelcref{04-eqn-18} implies \begin{align} \label{04-eqn-23} f(-\mathsf{g}) &= f(\mathsf{g}) + f(-\mathsf{e}) \end{align} for any $\mathsf{g} \in \mathrm{Q}_{8}$. \par Setting $(\mathsf{g}_{1},\mathsf{g}_{2}) = (\pm \mathsf{i}, \pm \mathsf{j}) , (\pm \mathsf{j}, \pm \mathsf{k}) , (\pm \mathsf{k}, \pm \mathsf{i})$ in \labelcref{04-eqn-18} (where the signs can vary independently of each other) all impose the condition \begin{align} \label{04-eqn-24} s_{2} &= 0 \end{align} for $s,s_{2}$ (check the case $(\mathsf{g}_{1},\mathsf{g}_{2}) = (\mathsf{i}, \mathsf{j})$, then use \labelcref{04-eqn-23} to show that changing the signs don't give new relations, then use \labelcref{04-eqn-22} to show that one can permute using left multiplication by $\mathsf{M}_{2}$). \par Setting $(\mathsf{g}_{1},\mathsf{g}_{2}) = (\pm \mathsf{j}, \pm \mathsf{i}) , (\pm \mathsf{k}, \pm \mathsf{j}) , (\pm \mathsf{i}, \pm \mathsf{k})$ in \labelcref{04-eqn-18} (where the signs can vary independently of each other) all impose the condition \begin{align} \label{04-eqn-25} s &= s_{3} \end{align} for $s_{3}$ (check the case $(\mathsf{g}_{1},\mathsf{g}_{2}) = (\mathsf{j}, \mathsf{i})$, then use \labelcref{04-eqn-23} to show that changing the signs don't give new relations, then use \labelcref{04-eqn-22} to show that one can permute using left multiplication by $\mathsf{M}_{2}$). \par Setting $(\mathsf{g}_{1},\mathsf{g}_{2}) = (\pm \mathsf{g}, \pm\mathsf{g})$ for $\mathsf{g} = \mathsf{i} , \mathsf{j} , \mathsf{k}$ (where the signs can vary independently of each other) all impose the condition \begin{align} \label{04-eqn-26} s = s_{2}+s_{3} \end{align} on $s,s_{2},s_{3}$ (check the case $\mathsf{g} = \mathsf{i}$, then use \labelcref{04-eqn-23} to show that changing the signs don't give new relations, then use \labelcref{04-eqn-22} to show that one can permute using left multiplication by $\mathsf{M}_{2}$), but \labelcref{04-eqn-26} is implied by \labelcref{04-eqn-24} and \labelcref{04-eqn-25}. \par These are the only relations satisfied by the $s,s_{1},s_{2},s_{3}$. Thus we have \[ \ker ((d_{1})^{\mathsf{M}_{2}}) \simeq \Z/(2) \oplus \Z/(2) \] since there are no relations on $s,s_{1} \in \Z/(2)$. \par An element of $(\mr{C}^{0}(\mathrm{Q}_{8},N_{1}))^{\mathsf{M}_{2}}$ corresponds to an element $(t,t,t) \in N_{1}$; since every element of $\SL_{2,3}$ fixes elements of this form (see \labelcref{04-eqn-15}), the image of $(t,t,t)$ under $(d_{0})^{\mathsf{M}_{2}}$ corresponds to the function $f : \mathrm{Q}_{8} \to N_{1}$ sending every element to $(0,0,0)$, in other words \[ \im ((d_{1})^{\mathsf{M}_{2}}) = 0 \] which implies \begin{align} \label{04-eqn-10} \H^{0}(\langle \mathsf{M}_{2} \rangle , \H^{1}(\mathrm{Q}_{8},N_{1})) \simeq \Z/(2) \oplus \Z/(2) \end{align} and so \begin{align} \label{04-eqn-11} \Br \ms{M}_{1,1,k} = \H_{\on{fppf}}^{2}(\ms{M}_{1,1,k} , \G_{m})[2] = \Z/(2) \end{align} by combining \labelcref{04-eqn-10} with \labelcref{04-eqn-19}, \labelcref{04-eqn-13}, \labelcref{04-eqn-12}, and \labelcref{04-eqn-02}. \qed \end{pg}

\begin{remark}[The inhomogeneous cochain complex admits a left $G$-action] \label{06} Let $G$ be a group, let $H \trianglelefteq G$ be a normal subgroup, and let $M$ be a left $G$-module. Set $P_{i} := \Z[H^{i+1}]$; we denote by $[\mathsf{h}_{0},\dotsc,\mathsf{h}_{i}]$ the canonical $\Z$-basis of $P_{i}$. We view $P_{i}$ as a left $H$-module via the diagonal action $\mathsf{h} \cdot [\mathsf{h}_{0},\dotsc,\mathsf{h}_{i}] = [\mathsf{h}\mathsf{h}_{0} , \dotsc , \mathsf{h}\mathsf{h}_{i}]$; then $P_{i}$ is a free left $\Z[H]$-module with basis consisting of elements of the form $[\mathsf{e},\mathsf{h}_{1},\dotsc,\mathsf{h}_{i}]$. Applying the functor $\Hom_{H}(-,M)$ to the bar resolution \[ \dotsb \to P_{2} \to P_{1} \to P_{0} \to \Z \to 0 \] gives the usual homogeneous cochain complex \[ \Hom_{\Z[H]}(P_{0} , M) \stackrel{\delta_{0}}{\to} \Hom_{\Z[H]}(P_{1} , M) \stackrel{\delta_{1}}{\to} \Hom_{\Z[H]}(P_{2} , M) \to \dotsb \] whose cohomology gives $\H^{i}(H,M)$. \par We note that there is a natural left $G$-action on $\Hom_{\Z[H]}(P_{i},M)$ for which the differential $\delta_{i} : \Hom_{\Z[H]}(P_{i},M) \to \Hom_{\Z[H]}(P_{i+1},M)$ is $G$-linear. Namely, the action of $\mathsf{g} \in G$ on $\varphi_{i} \in \Hom_{\Z[H]}(P_{i},M)$ is described by \[ (\mathsf{g}\varphi_{i})([\mathsf{h}_{0},\dotsc,\mathsf{h}_{i}]) := \mathsf{g} \cdot (\varphi_{i}([\mathsf{g}^{-1}\mathsf{h}_{0}\mathsf{g} , \dotsc , \mathsf{g}^{-1}\mathsf{h}_{i}\mathsf{g}])) \] for all $\mathsf{h}_{0},\dotsc,\mathsf{h}_{i} \in H$. Let \[ \mr{C}^{i}(H,M) := \mr{Fun}(H^{i},M) \] denote the abelian group of functions $H^{i} \to M$. Via the usual abelian group isomorphism \[ \Hom_{\Z[H]}(P_{i},M) \simeq \mr{C}^{i}(H,M) \] sending $\varphi_{i} \mapsto \{(\mathsf{h}_{1},\dotsc,\mathsf{h}_{i}) \mapsto \varphi_{i}(\mathsf{e},\mathsf{h}_{1},\mathsf{h}_{1}\mathsf{h}_{2},\dotsc,\mathsf{h}_{1} \dotsb \mathsf{h}_{i})\}$, the abelian group $\mr{C}^{i}(H,M)$ inherits a left action of $G$ described by \begin{align} \label{06-eqn-01} (\mathsf{g}f_{i})(\mathsf{h}_{1},\dotsc,\mathsf{h}_{i}) = \mathsf{g} \cdot (f_{i}(\mathsf{g}^{-1}\mathsf{h}_{1}\mathsf{g} , \dotsc , \mathsf{g}^{-1}\mathsf{h}_{i}\mathsf{g})) \end{align} for $\mathsf{g} \in G$ and $f_{i} \in \mr{C}^{i}(H,M)$. The inhomogeneous cochain complex \[ \mr{C}^{0}(H,M) \stackrel{d_{0}}{\to} \mr{C}^{1}(H,M) \stackrel{d_{1}}{\to} \mr{C}^{2}(H,M) \to \dotsb \] is $G$-linear as well. \par For $f_{0} \in \mr{C}^{0}(H,M)$, we have $(d_{0}f_{0})(\mathsf{h}_{1}) = \mathsf{h}_{1} \cdot f_{0}(\mathsf{e}) - f_{0}(\mathsf{e})$. \par For $f_{1} \in \mr{C}^{1}(H,M)$, we have $(d_{1}f_{1})(\mathsf{h}_{1},\mathsf{h}_{2}) = \mathsf{h}_{1} \cdot f_{1}(\mathsf{h}_{2}) - f_{1}(\mathsf{h}_{1}\mathsf{h}_{2}) + f_{1}(\mathsf{h}_{1})$. \par Let $\Sigma := G/H$ be the quotient; then there is an induced left action of $\Sigma$ on the cohomology $\mathrm{h}^{i}(\mr{C}^{\bullet}(H,M))$. In case $G \to \Sigma$ has a section, in which case $G$ is the semi-direct product $G \simeq H \rtimes \Sigma$, then this $\Sigma$-action coincides with the one obtained by restricting the $G$-action on $\mr{C}^{\bullet}(H,M)$ to $\Sigma$. \end{remark}

\begin{remark} \label{08} The arguments used in \Cref{04-04} and \Cref{04-03} are similar to those of Mathew and Stojanoska \cite[Appendix B]{MATHEW-STOJANOSKA-TPGOTMFVDT}, who show $\H^{1}(\GL_{2,3},(TMF(3)_{0})^{\times}) = \Z/(12)$ where $\GL_{2,3}$ acts on \begin{align} \label{08-eqn-01} \textstyle TMF(3)_{0} = \Z[\frac{1}{3} , \zeta , t , \frac{1}{t} , \frac{1}{1-\zeta t} , \frac{1}{1+\zeta^{2}t}]/(\zeta^{2}+\zeta+1) \end{align} as in \cite[\S4.3]{STOJANOSKA-CDF2PTMF}. \end{remark}

\begin{note}[Explicit description of inhomogeneous 1-cocycles] \label{22} We describe the 1-cocycles $\GL_{2,3} \to M$ obtained via the compositions \labelcref{04-eqn-19} and \labelcref{04-eqn-13}. By our computation in \Cref{04-03}, the 1-cocycles \[ f_{\mr{Q}_{8}} : \mr{Q}_{8} \to N_{1} \] are of the form \begin{align*} \mathsf{e} &\mapsto (0,0,0) & -\mathsf{e} &\mapsto (s,s,s) \\ \mathsf{i} &\mapsto (s_{1},0,s) & -\mathsf{i} &\mapsto (s_{1}+s,s,0) \\ \mathsf{j} &\mapsto (0,s,s_{1}) & -\mathsf{j} &\mapsto (s,0,s_{1}+s) \\ \mathsf{k} &\mapsto (s,s_{1},0) & -\mathsf{k} &\mapsto (0,s_{1}+s,s) \end{align*} for some $s,s_{1} \in \Z/(2)$. Suppose \[ f_{\SL_{2,3}} : \SL_{2,3} \to N_{1} \] is a 1-cocycle such that $f_{\SL_{2,3}}$ is fixed by the action of $\mathsf{M}_{2}$ (see \labelcref{06-eqn-01}) and which satisfies $f_{\SL_{2,3}}(\mathsf{g}) = f_{\mr{Q}_{8}}(\mathsf{g})$ for $\mathsf{g} \in \mr{Q}_{8}$. We have \[ \mathsf{M}_{2} \cdot f_{\SL_{2,3}}(\mathsf{M}_{2}^{-1} \cdot \mathsf{g} \cdot \mathsf{M}_{2}) = f_{\SL_{2,3}}(\mathsf{g}) \] for all $\mathsf{g} \in \SL_{2,3}$; taking $\mathsf{g} = \mathsf{M}_{2}$ gives $\mathsf{M}_{2} \cdot f_{\SL_{2,3}}(\mathsf{M}_{2}) = f_{\SL_{2,3}}(\mathsf{M}_{2})$. Taking $\mathsf{g}_{1} = \mathsf{g}_{2} = \mathsf{M}_{2}$ in the 1-cocycle condition \labelcref{04-eqn-18} then gives $f_{\SL_{2,3}}(\mathsf{M}_{2}) = 0$. Thus we have \begin{align} \label{22-eqn-01} f_{\SL_{2,3}}(\mathsf{g} \cdot \mathsf{M}_{2}) = f_{\SL_{2,3}}(\mathsf{g}) \end{align} for any $\mathsf{g} \in \SL_{2,3}$, again by \labelcref{04-eqn-18}. \par By Shapiro's lemma \labelcref{04-eqn-13}, there is a 1-cocycle \[ f_{\GL_{2,3}} : \GL_{2,3} \to \mr{Ind}_{\SL_{2,3}}^{\GL_{2,3}}(N_{1}) \] such that precomposing with the inclusion $\SL_{2,3} \subset \GL_{2,3}$ and postcomposing with the projection $\mr{Ind}_{\SL_{2,3}}^{\GL_{2,3}}(N_{1}) \to N_{1}$ gives $f_{\SL_{2,3}}$. After altering $f_{\GL_{2,3}}$ by a 1-coboundary, we may assume by \Cref{20151029-05} that $f_{\GL_{2,3}}$ is given by the formula \labelcref{20151029-05-eqn-01}, namely \begin{align} \label{22-eqn-02} f_{\GL_{2,3}}(\mathsf{g} \cdot \mathsf{M}_{1}^{i})([\mathsf{M}_{1}^{j}]) := f_{\SL_{2,3}}(\mathsf{M}_{1}^{j} \cdot \mathsf{g} \cdot \mathsf{M}_{1}^{-j}) \end{align} for any $i,j \in \{0,1\}$ and $\mathsf{g} \in \SL_{2,3}$. Any element $\mathsf{g} \in \GL_{2,3}$ may be expressed in the form \[ \mathsf{h} \cdot \mathsf{M}_{2}^{i_{2}} \cdot \mathsf{M}_{1}^{i_{1}} \] where $i_{1} \in \{0,1\}$ and $i_{2} \in \{0,1,2\}$ and $\mathsf{h} \in \mathrm{Q}_{8}$. We have formulas \begin{align} \label{22-eqn-03} \begin{aligned} \mathsf{M}_{1} \cdot \mathsf{M}_{2}^{-1} \cdot \mathsf{M}_{1} &= \mathsf{M}_{2}^{-1} \\ \mathsf{M}_{1} \cdot \mathsf{i} \cdot \mathsf{M}_{1}^{-1} &= -\mathsf{i} \\ \mathsf{M}_{1} \cdot \mathsf{j} \cdot \mathsf{M}_{1}^{-1} &= -\mathsf{k} \\ \mathsf{M}_{1} \cdot \mathsf{k} \cdot \mathsf{M}_{1}^{-1} &= -\mathsf{j} \end{aligned} \end{align} and so \begin{align*} f_{\GL_{2,3}}(\mathsf{h} \cdot \mathsf{M}_{2}^{i_{2}} \cdot \mathsf{M}_{1}^{i_{1}})([\mathsf{M}_{1}^{j}]) &\stackrel{1}{=} f_{\SL_{2,3}}(\mathsf{M}_{1}^{j} \cdot \mathsf{h} \cdot \mathsf{M}_{2}^{i_{2}} \cdot \mathsf{M}_{1}^{-j}) \\ &= f_{\SL_{2,3}}((\mathsf{M}_{1}^{j} \cdot \mathsf{h} \cdot \mathsf{M}_{1}^{-j}) \cdot (\mathsf{M}_{1}^{j} \cdot \mathsf{M}_{2}^{i_{2}} \cdot \mathsf{M}_{1}^{-j})) \\ &\stackrel{2}{=} f_{\SL_{2,3}}(\mathsf{M}_{1}^{j} \cdot \mathsf{h} \cdot \mathsf{M}_{1}^{-j}) \\ &\stackrel{3}{=} f_{\mr{Q}_{8}}(\mathsf{M}_{1}^{j} \cdot \mathsf{h} \cdot \mathsf{M}_{1}^{-j}) \end{align*} where equality 1 is by \labelcref{22-eqn-02} and equality 2 is by \labelcref{22-eqn-01} and \labelcref{22-eqn-03} and equality 3 is since $\mathsf{M}_{1}^{j} \cdot \mathsf{h} \cdot \mathsf{M}_{1}^{-j} \in \mr{Q}_{8}$ (see \labelcref{22-eqn-03}). This is summarized in \labelcref{22-eqn-04} below. \begin{align} \label{22-eqn-04} \begin{aligned} f_{\GL_{2,3}}(\mathsf{e}) &= ( f_{\mr{Q}_{8}}(\mathsf{e}) , f_{\mr{Q}_{8}}(\mathsf{e}) ) = ((0,0,0),(0,0,0)) \\ f_{\GL_{2,3}}(\mathsf{i}) &= ( f_{\mr{Q}_{8}}(\mathsf{i}) , f_{\mr{Q}_{8}}(-\mathsf{i}) ) = ((s_{1},0,s),(s_{1}+s,s,0)) \\ f_{\GL_{2,3}}(\mathsf{j}) &= ( f_{\mr{Q}_{8}}(\mathsf{j}) , f_{\mr{Q}_{8}}(-\mathsf{k}) ) = ((0,s,s_{1}),(0,s_{1}+s,s)) \\ f_{\GL_{2,3}}(\mathsf{k}) &= ( f_{\mr{Q}_{8}}(\mathsf{k}) , f_{\mr{Q}_{8}}(-\mathsf{j}) ) = ((s,s_{1},0),(s,0,s_{1}+s)) \end{aligned} \end{align} \end{note}

\begin{note}[The Shapiro isomorphism and inhomogeneous 1-cocycles] \label{20151029-05} \footnote{Ehud Meir's MathOverflow post \cite{MEIR-SLITLOGE} was helpful in working out the details of this section.} Let $G$ be a group, let $H \subseteq G$ be a normal subgroup of finite index such that the projection $G \to G/H$ has a section $G/H \to G$ whose image corresponds to a subgroup $\Sigma$ of $G$. Let $N$ be a left $H$-module and let $\mr{Ind}_{H}^{G}N := \Hom_{\Z[H]}(\Z[G],N)$ denote the associated induced left $G$-module. We recall that the left $G$-action on $\mr{Ind}_{H}^{G}N$ sends $\varphi \mapsto \mathsf{g}\varphi$ where $(\mathsf{g}\varphi)(x) = \varphi(x\mathsf{g})$. \par We describe the inverse of the Shapiro isomorphism $\H^{1}(G,\mr{Ind}_{H}^{G}N) \to \H^{1}(H,N)$ in terms of inhomogeneous cochains. Suppose given a function \[ f : H \to N \] which satisfies \[ f(\mathsf{h}_{1}\mathsf{h}_{2}) = \mathsf{h}_{1} \cdot f(\mathsf{h}_{2}) + f(\mathsf{h}_{1}) \] for all $\mathsf{h}_{1},\mathsf{h}_{2} \in H$. We construct a 1-cocycle \[ s : G \to \mr{Ind}_{H}^{G}(N) \] which restricts to $f$, i.e. satisfies $s(\mathsf{h})(1 \cdot [\mathsf{e}]) = f(\mathsf{h})$ for all $\mathsf{h} \in H$. Note that every element of $\mathsf{g} \in G$ may be written uniquely in the form \[ \mathsf{g} = \mathsf{h}\sigma \] for $\mathsf{h} \in H$ and $\sigma \in \Sigma$, hence the collection $\{[\sigma]\}_{\sigma \in \Sigma}$ forms a basis for $\Z[G]$ as a left $\Z[H]$-module. We set \begin{align} \label{20151029-05-eqn-01} s(\mathsf{h}\sigma)([\xi]) := f(\xi\mathsf{h}\xi^{-1}) \end{align} for $\mathsf{h} \in H$ and $\sigma,\xi \in \Sigma$ and extend $\Z[H]$-linearly. Given $\mathsf{g}_{1},\mathsf{g}_{2} \in G$ where $\mathsf{g}_{i} = \mathsf{h}_{i}\sigma_{i}$ with $\mathsf{h}_{i} \in H$ and $\sigma_{i} \in \Sigma$, for any $\xi \in \Sigma$ we have \begin{align*} s(\mathsf{g}_{1}\mathsf{g}_{2})([\xi]) &= s(\mathsf{h}_{1}\sigma_{1}\mathsf{h}_{2}\sigma_{2})([\xi]) \\ &= s(\mathsf{h}_{1}(\sigma_{1}\mathsf{h}_{2}\sigma_{1}^{-1})\sigma_{1}\sigma_{2})([\xi]) \\ &= f(\xi\mathsf{h}_{1}(\sigma_{1}\mathsf{h}_{2}\sigma_{1}^{-1})\xi^{-1}) \end{align*} and \begin{align*} (\mathsf{g}_{1} \cdot s(\mathsf{g}_{2}))([\xi]) &= s(\mathsf{h}_{2}\sigma_{2})([\xi\mathsf{h}_{1}\sigma_{1}]) \\ &= s(\mathsf{h}_{2}\sigma_{2})([(\xi\mathsf{h}_{1}\xi^{-1})\xi\sigma_{1}]) \\ &= (\xi\mathsf{h}_{1}\xi^{-1}) \cdot s(\mathsf{h}_{2}\sigma_{2})([\xi\sigma_{1}]) \\ &= (\xi\mathsf{h}_{1}\xi^{-1}) \cdot f((\xi\sigma_{1})\mathsf{h}_{2}(\xi\sigma_{1})^{-1}) \end{align*} and \begin{align*} s(\mathsf{g}_{1})([\xi]) &= s(\mathsf{h}_{1}\sigma_{1})([\xi]) = f(\xi\mathsf{h}_{1}\xi^{-1}) \end{align*} which implies \[ s(\mathsf{g}_{1}\mathsf{g}_{2}) = \mathsf{g}_{1} \cdot s(\mathsf{g}_{2}) + s(\mathsf{g}_{1}) \] by $\Z[H]$-linearity and since $f$ is a 1-cocycle; hence $s$ is a 1-cocycle. \qed \end{note}

\begin{pg}[{Proof of \Cref{20170322-02}}] Let $k^{\mr{sep}}$ be a fixed separable closure of $k$ and let $\mr{G}_{k} := \mr{Gal}(k^{\mr{sep}}/k) \simeq \widehat{\Z}$ be the absolute Galois group. Set $\ms{M} := \ms{M}_{1,1,k}$ and $\ms{M}^{\mr{sep}} := \ms{M}_{1,1,k^{\mr{sep}}}$. We have $\Br \ms{M} = \Br' \ms{M}$ by \Cref{20170321-05}. The Leray spectral sequence for the map $\ms{M} \to \Spec k$ is of the form \[ \mr{E}_{2}^{p,q} = \H^{p}(\mr{G}_{k} , \H^{q}_{\et}(\ms{M}^{\mr{sep}} , \G_{m})) \implies \H^{p+q}_{\et}(\ms{M} , \G_{m}) \] with differentials $\mr{E}_{2}^{p,q} \to \mr{E}_{2}^{p+2,q-1}$. Here we have $\Gamma(\ms{M}^{\mr{sep}},\G_{m}) = \Gamma(\A_{k^{\mr{sep}}}^{1},\G_{m}) = (k^{\mr{sep}})^{\times}$ since $\ms{M}^{\mr{sep}} \to \A_{k^{\mr{sep}}}^{1}$ is the coarse moduli space map. Since $k$ is a finite field, we have that $\H^{0}_{\et}(\ms{M}^{\mr{sep}},\G_{m})$ is a torsion group. Moreover $\H^{1}_{\et}(\ms{M}^{\mr{sep}} , \G_{m}) \simeq \Pic(\ms{M}^{\mr{sep}}) \simeq \Z/(12)$ is a torsion group by \cite{FULTON-OLSSON-PICARD}. Thus by e.g. \cite[4.3.7]{LEIFU} or \cite[6.1.3]{GILLE-SZAMUELY} we have $\mr{E}_{2}^{p,q} = 0$ for $(p,q) \in \Z_{\ge 2} \times \{0,1\}$. This means there is an exact sequence \begin{align} \label{20170322-02-eqn-01} 0 \to \mr{E}_{2}^{1,1} \to \H^{2}_{\et}(\ms{M} , \G_{m}) \to \mr{E}_{2}^{0,2} \to 0 \end{align} of abelian groups. \par By \cite{FULTON-OLSSON-PICARD}, we have that $\Pic(\ms{M}^{\mr{sep}}) \simeq \Z/(12)$ is generated by the class of the Hodge bundle; since $\mr{G}_{k}$ acts trivially on invariant differentials of elliptic curves $E \to S$ where $S$ is a $k$-scheme, the action of $\mr{G}_{k}$ on $\Pic(\ms{M}^{\mr{sep}})$ is trivial. Hence we have \[ \mr{E}_{2}^{1,1} = \H^{1}(\mr{G}_{k} , \H^{1}_{\et}(\ms{M}^{\mr{sep}} , \G_{m})) \stackrel{1}{=} \Hom_{\mr{cont}}(\mr{G}_{k} , \Pic(\ms{M}^{\mr{sep}})) \stackrel{2}{=} \Z/(12) \] where equality 1 is by \cite[4.3.7]{LEIFU} and equality 2 is since $\mr{G}_{k} \simeq \widehat{\Z}$. We have \[ \mr{E}_{2}^{0,2} = \H^{0}(\mr{G}_{k} , \H^{2}_{\et}(\ms{M}^{\mr{sep}} , \G_{m})) \stackrel{1}{=} (\Z/(2))^{\mr{G}_{k}} \stackrel{2}{=} \Z/(2) \] where equality 1 is by the computation for an algebraically closed field (\Cref{48}) and also the fact that $\H^{2}_{\et}(\ms{M}^{\mr{sep}} , \G_{m})$ is a torsion group (see \cite[Proposition 2.5 (iii)]{ANTIEAU-MEIER-TBGOTMSOEC}) and equality 2 is because any group action on the group of order $2$ is necessarily trivial. Thus \labelcref{20170322-02-eqn-01} reduces to a natural extension \begin{align} \label{20170322-02-eqn-02} 0 \to \Z/(12) \to \Br \ms{M} \to \Z/(2) \to 0 \end{align} and it remains to see whether \labelcref{20170322-02-eqn-02} is split. It suffices to compute the size of $(\Br \ms{M})[2]$, since $(\Br \ms{M})[2]$ has $4$ or $2$ elements depending on whether \labelcref{20170322-02-eqn-02} is split or not, respectively. \par As in \Cref{04-02}, the fppf Kummer sequence \begin{align} \label{20170322-05-eqn-20} 1 \to \mu_{2} \to \G_{m} \stackrel{\times 2}{\to} \G_{m} \to 1 \end{align} gives an exact sequence \begin{align} \label{20170322-05-eqn-02} 1 \to \Z/(2) \stackrel{\partial}{\to} \H_{\on{fppf}}^{2}(\ms{M} , \mu_{2}) \to (\Br \ms{M})[2] \to 1 \end{align} of abelian groups. We compute $\H_{\on{fppf}}^{2}(\ms{M} , \mu_{2})$ using the Leray spectral sequence which is of the form \[ \mr{E}_{2}^{p,q} = \H^{p}(\mr{G}_{k},\H^{q}_{\mr{fppf}}(\ms{M}^{\mr{sep}} , \mu_{2})) \implies \H^{p+q}_{\mr{fppf}}(\ms{M} , \mu_{2}) \] with differentials $\mr{E}_{2}^{p,q} \to \mr{E}_{2}^{p+2,q-1}$. We have \[ \H^{p}_{\mr{fppf}}(\ms{M}^{\mr{sep}},\mu_{2}) = \begin{cases} 0 &\text{if }p = 0 \\ \Z/(2) &\text{if }p = 1 \\ \Z/(2) \oplus \Z/(2) &\text{if }p = 2 \end{cases} \] from the fppf Kummer sequence on $\ms{M}^{\mr{sep}}$, where the $p=0$ case follows since we are in characteristic $2$ and $\Gamma(\ms{M}^{\mr{sep}},\G_{m}) = \Gamma(\A_{k^{\mr{sep}}}^{1},\G_{m}) = (k^{\mr{sep}})^{\times}$, the $p=1$ case is since the multiplication-by-2 map on $\Gamma(\ms{M}^{\mr{sep}},\G_{m}) = (k^{\mr{sep}})^{\times}$ is an isomorphism, and the $p=2$ case is by the computation in the algebraically closed case (combine \labelcref{04-eqn-12}, \labelcref{04-eqn-13}, \labelcref{04-eqn-19}, \labelcref{04-eqn-10}). \par Since $k$ has characteristic 2, the 2-cohomological dimension of $k$ satisfies $\mr{cd}_{2}(k) \le 1$ by e.g. \cite[6.1.9]{GILLE-SZAMUELY}; hence $\mr{E}_{2}^{p,q} = 0$ for $p \ge 2$ and any $q$. Hence there is an exact sequence \begin{align} \label{20170322-02-eqn-04} 0 \to \H^{1}(\mr{G}_{k} , \H^{1}_{\mr{fppf}}(\ms{M}^{\mr{sep}} , \mu_{2})) \to \H^{2}_{\mr{fppf}}(\ms{M} , \mu_{2}) \to \H^{0}(\mr{G}_{k} , \H^{2}_{\mr{fppf}}(\ms{M}^{\mr{sep}} , \mu_{2})) \to 0 \end{align} of abelian groups. As above, the $\mr{G}_{k}$-action on $\H^{1}_{\mr{fppf}}(\ms{M}^{\mr{sep}} , \mu_{2})$ is necessarily trivial so we have an isomorphism $\H^{1}(\mr{G}_{k} , \H^{1}_{\mr{fppf}}(\ms{M}^{\mr{sep}} , \mu_{2})) \simeq \Hom_{\mr{cont}}(\mr{G}_{k} , \Z/(2)) \simeq \Z/(2)$. \par To describe $\H^{0}(\mr{G}_{k} , \H^{2}_{\mr{fppf}}(\ms{M}^{\mr{sep}} , \mu_{2}))$, we describe the $\mr{G}_{k}$-action on $\H^{2}_{\mr{fppf}}(\ms{M}^{\mr{sep}} , \mu_{2})$. Let \[ \xi \in k^{\mr{sep}} \] be a fixed root of $x^{2}+x+1$ (i.e. a primitive 3rd root of unity). \par If $\xi \in k$, then $\mr{G}_{k}$ acts trivially on $\H^{2}_{\mr{fppf}}(\ms{M}^{\mr{sep}} , \mu_{2})$; hence $\H^{0}(\mr{G}_{k} , \H^{2}_{\mr{fppf}}(\ms{M}^{\mr{sep}} , \mu_{2}))$ has $4$ elements, hence $\H^{2}_{\mr{fppf}}(\ms{M} , \mu_{2})$ has $8$ elements by \labelcref{20170322-02-eqn-04}, hence $(\Br \ms{M})[2]$ has $4$ elements by \labelcref{20170322-05-eqn-02}, hence $\Br \ms{M} \simeq \Z/(2) \oplus \Z/(12)$. \par Suppose $\xi \not\in k$. The $k$-algebra map \[ \textstyle k[\mu,\omega,\frac{1}{\mu^{3}-1}]/(\omega^{2}+\omega+1) \to k^{\mr{sep}}[\nu_{1},\frac{1}{\nu_{1}^{3}-1}] \times k^{\mr{sep}}[\nu_{2},\frac{1}{\nu_{2}^{3}-1}] \] sending $\mu \mapsto (\nu_{1},\nu_{2})$ and $\omega \mapsto (\xi,\xi^{2})$ induces an isomorphism \begin{align} \label{20170322-05-eqn-10} \textstyle k[\mu,\omega,\frac{1}{\mu^{3}-1}]/(\omega^{2}+\omega+1) \otimes_{k} k^{\mr{sep}} \to k^{\mr{sep}}[\nu_{1},\frac{1}{\nu_{1}^{3}-1}] \times k^{\mr{sep}}[\nu_{2},\frac{1}{\nu_{2}^{3}-1}] \end{align} of $k^{\mr{sep}}$-algebras. The inverse to \labelcref{20170322-05-eqn-10} sends \[ ( f_{1}(\nu_{1}) , f_{2}(\nu_{2}) ) \mapsto f_{1}(\mu) \left( \omega \otimes \frac{1}{\xi-\xi^{2}} + 1 \otimes \frac{\xi}{\xi-1} \right) + f_{2}(\mu) \left( (-\omega) \otimes \frac{1}{\xi-\xi^{2}} + (-1) \otimes \frac{1}{\xi-1} \right) \] for $f_{i}(\nu_{i}) \in k[\nu_{i} , \frac{1}{\nu_{i}^{3}-1}]$. \par Let \[ \lambda \in \mr{G}_{k} \] be an automorphism of $k^{\mr{sep}}$ such that $\lambda(\xi) = \xi^{2}$. Then the $k$-algebra automorphism of $k^{\mr{sep}}[\nu_{1},\frac{1}{\nu_{1}^{3}-1}] \times k^{\mr{sep}}[\nu_{2},\frac{1}{\nu_{2}^{3}-1}]$ induced by \labelcref{20170322-05-eqn-10} sends $(\nu_{1} , 0) \mapsto (0 , \nu_{2})$ and $(0, \nu_{2}) \mapsto (\nu_{1} , 0)$ and $(\xi , 0) \mapsto (0 , \xi^{2})$ and $(0 , \xi) \mapsto (\xi^{2} , 0)$. We see that the action of $\lambda$ on $M$ (see \labelcref{04-eqn-07}) is given by \labelcref{20170322-02-eqn-03}. \begin{equation} \label{20170322-02-eqn-03} {\small \begin{array}{r|c|c|c|c|c|c} & \nu_{1}-1 & \nu_{1}-\xi & \nu_{1}-\xi^{2} & \nu_{2}-1 & \nu_{2}-\xi & \nu_{2}-\xi^{2} \\ \lambda & \nu_{2}-1 & \nu_{2}-\xi^{2} & \nu_{2}-\xi & \nu_{1}-1 &\nu_{1}-\xi^{2} & \nu_{1}-\xi \end{array} } \end{equation} A computation with \labelcref{20170322-02-eqn-03} and \labelcref{04-eqn-08} shows that \begin{align} \label{20170322-02-eqn-05} \lambda \mathsf{g} \lambda^{-1} \cdot m = \mathsf{g} \cdot m \end{align} for any $m \in M$ and $\mathsf{g} \in \GL_{2,3}$. \par Let $f_{\GL_{2,3}} : \GL_{2,3} \to M$ be an inhomogeneous 1-cocycle as in \Cref{22}. Multiplying the 1-cocycle condition \labelcref{04-eqn-18} on the left by $\lambda$ gives \begin{align*} \lambda \cdot f_{\GL_{2,3}}(\mathsf{g}_{1} \cdot \mathsf{g}_{2}) &= \lambda\mathsf{g}_{1} \cdot f_{\GL_{2,3}}(\mathsf{g}_{2}) + \lambda \cdot f_{\GL_{2,3}}(\mathsf{g}_{1}) \\ &\stackrel{1}{=} \mathsf{g}_{1} \cdot (\lambda \cdot f_{\GL_{2,3}}(\mathsf{g}_{2})) + \lambda \cdot f_{\GL_{2,3}}(\mathsf{g}_{1}) \end{align*} where equality 1 follows from \labelcref{20170322-02-eqn-05}. Hence the function $\lambda \cdot f_{\GL_{2,3}} : \GL_{2,3} \to M$ sending $\mathsf{g} \mapsto \lambda \cdot f_{\GL_{2,3}}(\mathsf{g})$ is a 1-cocycle as well. Using \labelcref{20170322-02-eqn-03} and \labelcref{22-eqn-04}, we have that \begin{align} \label{20170322-02-eqn-06} \begin{aligned} (\lambda \cdot f_{\GL_{2,3}})(\mathsf{e}) &= ((0,0,0),(0,0,0)) \\ (\lambda \cdot f_{\GL_{2,3}})(\mathsf{i}) &= ((s_{1}+s,0,s),(s_{1},s,0)) \\ (\lambda \cdot f_{\GL_{2,3}})(\mathsf{j}) &= ((0,s,s_{1}+s),(0,s_{1},s)) \\ (\lambda \cdot f_{\GL_{2,3}})(\mathsf{k}) &= ((s,s_{1}+s,0),(s,0,s_{1})) \end{aligned} \end{align} and so \begin{align} \label{20170322-02-eqn-08} \begin{aligned} f_{\GL_{2,3}}(\mathsf{e}) - (\lambda \cdot f_{\GL_{2,3}})(\mathsf{e}) &= ((0,0,0),(0,0,0)) \\ f_{\GL_{2,3}}(\mathsf{i}) - (\lambda \cdot f_{\GL_{2,3}})(\mathsf{i}) &= ((s,0,0),(s,0,0)) \\ f_{\GL_{2,3}}(\mathsf{j}) - (\lambda \cdot f_{\GL_{2,3}})(\mathsf{j}) &= ((0,0,s),(0,s,0)) \\ f_{\GL_{2,3}}(\mathsf{k}) - (\lambda \cdot f_{\GL_{2,3}})(\mathsf{k}) &= ((0,s,0),(0,0,s)) \end{aligned} \end{align} for the same $s,s_{1} \in \Z/(2)$ as in \labelcref{22-eqn-04}. \par Suppose $f_{\GL_{2,3}}$ and $\lambda \cdot f_{\GL_{2,3}}$ differ by a 1-coboundary, in other words there exists an element \[ m := ((m_{1}^{1},m_{2}^{1},m_{3}^{1}),(m_{1}^{2},m_{2}^{2},m_{3}^{2})) \in M \] such that \begin{align} \label{20170322-02-eqn-07} f_{\GL_{2,3}}(\mathsf{g}) - (\lambda \cdot f_{\GL_{2,3}})(\mathsf{g}) = \mathsf{g} \cdot m - m \end{align} for all $\mathsf{g} \in \GL_{2,3}$. By \labelcref{20170322-02-eqn-08}, taking $\mathsf{g} = \mathsf{M}_{2}$ in \labelcref{20170322-02-eqn-07} gives $m^{i} := m_{1}^{i} = m_{2}^{i} = m_{3}^{i}$ for $i=1,2$; then taking $\mathsf{g} = \mathsf{M}_{1}$ gives $m^{1} = m^{2}$; then taking $\mathsf{g} = \mathsf{i}$ gives $m = 0$. We see that $f_{\GL_{2,3}}$ and $\lambda \cdot f_{\GL_{2,3}}$ differ by a 1-coboundary if and only if $s = 0$. \par Hence we have that $\H^{0}(\mr{G}_{k} , \H^{2}_{\mr{fppf}}(\ms{M}^{\mr{sep}} , \mu_{2})) \simeq \Z/(2)$, hence $\H^{2}_{\mr{fppf}}(\ms{M} , \mu_{2})$ has $4$ elements by \labelcref{20170322-02-eqn-04}, hence $(\Br \ms{M})[2]$ has $2$ elements by \labelcref{20170322-05-eqn-02}, hence $\Br \ms{M} \simeq \Z/(24)$. \qed \end{pg}

\appendix

\section{The Weierstrass and Hesse presentations of $[\Gamma(3)]$} \label{sec-08}

The purpose of this section is to prove \Cref{24} below, which we could not find proved in the literature. For completeness of exposition, we first recall the definition of a full level $N$ structure on an elliptic curve $E/S$.

\begin{pg}[Full level $N$ structure] \label{20} \cite[Ch. 3]{KATZ-MAZUR-AMOEC} Let $N$ be a positive integer. We define $[\Gamma(N)]$ to be the category of pairs \[ (E/S , \xi) \] where \[ E/S = (f : E \to S , e : S \to E) \] is an elliptic curve and \[ \xi : (\Z/(N))^{2}_{S} \to E \] is a morphism of $S$-group schemes inducing an isomorphism $(\Z/(N))^{2}_{S} \simeq E[N]$. A morphism \[ (E_{1}/S_{1},\xi_{1}) \to (E_{2}/S_{2},\xi_{2}) \] is a pair \[ (\alpha : E_{1} \to E_{2} \;,\; \beta : S_{1} \to S_{2} ) \] of morphisms of schemes such that the diagram \begin{equation} \label{20-eqn-01} \begin{tikzpicture}[>=stealth, baseline=(current bounding box.center)] 
\matrix[matrix of math nodes,row sep=3em, column sep=1.5em, text height=1.5ex, text depth=0.25ex] { 
|[name=11]|  & |[name=12]| E_{1} & |[name=13]|  & |[name=14]| E_{2} \\ 
|[name=21]| (\Z/(N))^{2}_{S_{1}} & |[name=22]|  & |[name=23]| (\Z/(N))^{2}_{S_{2}} & |[name=24]|  \\ 
|[name=31]|  & |[name=32]| S_{1} & |[name=33]|  & |[name=34]| S_{2} \\ 
}; 
\draw[->,font=\scriptsize]
(12) edge node[above=-1pt] {$\alpha$} (14) (32) edge node[below=-1pt] {$\beta$} (34) (21) edge[pos=0.8] node[above=-1pt] {$\id \times \beta$} (23) (21) edge node[above left=-2pt] {$\xi_{1}$} (12) (23) edge node[above left=-2pt] {$\xi_{2}$} (14) (12) edge[pos=0.75] node[right=-1pt] {$f_{1}$} (32) (14) edge node[right=-1pt] {$f_{2}$} (34) (21) edge (32) (23) edge (34); \end{tikzpicture} \end{equation} commutes, where the morphism $\id \times \beta$ is the one induced by the identity on $(\Z/(3))^{2}_{\Z}$ and $\beta$, and such that $\alpha$ induces an isomorphism of $S_{1}$-group schemes $E_{1} \simeq S_{1} \times_{\beta,S_{2}} E_{2}$. \par There is a functor \[ [\Gamma(N)] \to \ms{M}_{1,1,\Z} \] sending $(E/S,\xi) \mapsto E/S$ on objects and $(\alpha,\beta) \mapsto (\alpha,\beta)$ on morphisms. If $E/S$ admits a full level $N$ structure, then $N$ is invertible on $S$ by \cite[2.3.2]{KATZ-MAZUR-AMOEC}, hence the above functor factors through $\ms{M}_{1,1,\Z[\frac{1}{N}]}$. If $N \ge 3$, then for any scheme $S$ the fiber category $[\Gamma(N)](S)$ is equivalent to a set by \cite[2.7.2]{KATZ-MAZUR-AMOEC}, so $[\Gamma(N)]$ is fibered in sets over the category of schemes. \end{pg}

\begin{pg}[The {$\GL_{2}(\Z/(N))$}-action on {$[\Gamma(N)]$}] \label{23} Fix a scheme $S$. For any element \[ \sigma = \begin{bmatrix} \sigma_{11} & \sigma_{12} \\ \sigma_{21} & \sigma_{22} \end{bmatrix} \] in $\GL_{2}(\Z/(N))$, let \[ \varphi_{\sigma} : (\Z/(N))^{2}_{S} \to (\Z/(N))^{2}_{S} \] be the $S$-group scheme automorphism of $(\Z/(N))^{2}_{S}$ corresponding to the abelian group homomorphism $(\Z/(N))^{2} \to (\Z/(N))^{2}$ defined by \[ \begin{bmatrix} x_{1} \\ x_{2} \end{bmatrix} \mapsto \begin{bmatrix} \sigma_{11} & \sigma_{12} \\ \sigma_{21} & \sigma_{22} \end{bmatrix} \begin{bmatrix} x_{1} \\ x_{2} \end{bmatrix} = \begin{bmatrix} \sigma_{11}x_{1} + \sigma_{12}x_{2} \\ \sigma_{21}x_{1} + \sigma_{22}x_{2} \end{bmatrix} \] for $x_{1},x_{2} \in \Z/(N)$, i.e. acting by multiplication on the left on $(\Z/(N))^{2}$ viewed as vertical vectors. We have \[ \varphi_{\sigma_{1}}\varphi_{\sigma_{2}} = \varphi_{\sigma_{1}\sigma_{2}} \] for $\sigma_{1},\sigma_{2} \in \GL_{2}(\Z/(N))$. \par Fix an object $(E/S,\xi) \in [\Gamma(N)](E/S)$; then $(E/S,\xi \circ \varphi_{\sigma})$ is another object of $[\Gamma(N)](E/S)$, i.e. corresponds to another full level $N$ structure on $E/S$. This implies that there is a natural action of $\GL_{2}(\Z/(N))$ on each fiber category $[\Gamma(N)](E/S)$; the action is a right action since it is defined by precomposition. \end{pg}

\begin{theorem} \label{21} \cite[4.7.2]{KATZ-MAZUR-AMOEC} If $N \ge 3$, the category $[\Gamma(N)]$ is representable by a smooth affine curve $Y(N)$ over $\Z[\frac{1}{N}]$. \end{theorem}

We are primarily interested in the case $N=3$. The $3$-torsion points of an elliptic curve correspond to its inflection points (also ``flex points''). In \cite[(2.2.11)]{KATZ-MAZUR-AMOEC} it is shown that $Y(3) \simeq \Spec A_{\on{W}}$ where \[ A_{\on{W}} := \textstyle \Z[\frac{1}{3},B,C,\frac{1}{C} , \frac{1}{a_{3}} , \frac{1}{a_{1}^{3}-27a_{3}}]/(B^{3}-(B+C)^{3}) \] and the universal elliptic curve over $A_{\on{W}}$ with full level 3 structure is the pair \begin{align} \label{20150206-30-eqn-15} \begin{cases} E_{\on{W}} := \Proj A_{\on{W}}[X,Y,Z]/(Y^{2}Z + a_{1}XYZ + a_{3}YZ^{2} = X^{3}) \\ [0:0:1] , [C:B+C:1] \end{cases} \end{align} where \begin{align} \label{20150206-30-eqn-03} a_{1} &= 3C-1 \\ \label{20150206-30-eqn-04} a_{3} &= -3C^{2}-B-3BC \;. \end{align} 
The formulas \labelcref{20150206-30-eqn-03} and \labelcref{20150206-30-eqn-04} are obtained by imposing the condition that the line $Y = X + BZ$ is a flex tangent to $E_{\on{W}}$ at $[C : B+C : 1]$. The ring $A_{\on{W}}$ is isomorphic to $TMF(3)_{0}$ \labelcref{08-eqn-01}, with mutually inverse ring isomorphisms $TMF(3)_{0} \to A_{\on{W}}$ and $A_{\on{W}} \to TMF(3)_{0}$ given by $(\zeta,t) \mapsto (\frac{B+C}{C} , \frac{1}{3C})$ and $(B,C) \mapsto (\frac{1}{3(\zeta-1)t} , \frac{1}{3t})$ respectively.

In this paper, however, we use the ``Hesse presentation'' of $Y(3)$ as in \cite[5.1]{FULTON-OLSSON-PICARD}. The following is claimed without proof in the Introduction to \cite{DELIGNE-RAPOPORT-LSDMDCE} and \cite[5.2.30]{HARDER-LOAG-I}.

\begin{proposition} \label{24} There is an isomorphism $Y(3) \simeq \Spec A_{\on{H}}$ where \[ A_{\on{H}} := \textstyle \Z[\frac{1}{3},\mu,\omega,\frac{1}{\mu^{3}-1}]/(\omega^{2}+\omega+1) \] and the universal elliptic curve over $A_{\on{H}}$ with full level 3 structure is the pair \begin{align} \label{20150206-30-eqn-16} \begin{cases} E_{\on{H}} := \Proj A_{\on{H}}[X,Y,Z]/(X^{3}+Y^{3}+Z^{3} = 3\mu XYZ) \\ [-1:0:1] , [1:-\omega:0] \end{cases} \end{align} with identity section $[1:-1:0]$. \end{proposition}

The explicit $\Z[\frac{1}{3}]$-algebra isomorphisms $A_{\on{H}} \to A_{\on{W}}$ and $A_{\on{W}} \to A_{\on{H}}$ are given in \labelcref{20150206-30-eqn-13} and \labelcref{20150206-30-eqn-14} respectively.

\begin{pg} \label{07} By \cite[\S4]{SMART-HESSIAN}, the group law of an elliptic curve $E = \Proj A[X,Y,Z]/(X^{3}+Y^{3}+Z^{3} = 3\mu XYZ)$ in Hessian form over a ring $A$ is as follows. If $P = [x:y:z]$, then $2P = [x':y':z']$ where \begin{align*} x' &= y(z^{3}-x^{3}) \\ y' &= x(y^{3}-z^{3}) \\ z' &= z(x^{3}-y^{3}) \end{align*} and if $P_{i} = [x_{i}:y_{i}:z_{i}]$ are points of $E_{\on{H}}$ for $i=1,2,3$ satisfying $P_{1} + P_{2} = P_{3}$, then \begin{align*} x_{3} &= x_{2}y_{1}^{2}z_{2} - x_{1}y_{2}^{2}z_{1} \\ y_{3} &= x_{1}^{2}y_{2}z_{2} - x_{2}^{2}y_{1}z_{1} \\ z_{3} &= x_{2}y_{2}z_{1}^{2} - x_{1}y_{1}z_{2}^{2} \end{align*} which only makes sense if $P_{1} \ne P_{2}$. \par Using the above formulas, we may check that the full level 3 structure $\xi_{\on{H}} : (\Z/(3))^{2}_{A_{\on{H}}} \to E_{\on{H}}$ is given by the table \labelcref{20150206-09-eqn-01}. \begin{equation} \label{20150206-09-eqn-01} \xi_{\on{H}} \left( \begin{bmatrix} (0,0) & (1,0) & (2,0) \\[0.5em] (0,1) & (1,1) & (2,1) \\[0.5em] (0,2) & (1,2) & (2,2) \end{bmatrix} \right) = \begin{bmatrix} [1:-1:0] & [-1:0:1] & [0:1:-1] \\[0.5em] [1:-\omega:0] & [-\omega:0:1] & [0:1:-\omega] \\[0.5em] [1:-\omega^{2}:0] & [-\omega^{2}:0:1] & [0:1:-\omega^{2}] \end{bmatrix} \end{equation} The Hesse presentation \labelcref{20150206-30-eqn-16} is sometimes easier to work with than the Weierstrass presentation \labelcref{20150206-30-eqn-15} since the equation of the universal elliptic curve is symmetric in $X,Y,Z$, which means that there is also considerable symmetry in the 3-torsion points \labelcref{20150206-09-eqn-01}. \end{pg}

\begin{pg} \label{25} We describe the $\GL_{2}(\Z/(3))$-action on $E_{\on{H}}/A_{\on{H}}$. Set $S_{\mr{H}} := \Spec A_{\mr{H}}$. The functor $[\Gamma(3)]$ being representable by $S_{\mr{H}}$ means explicitly that for any $\Z[\frac{1}{3}]$-scheme $T$ and object $(E/T,\xi) \in ([\Gamma(3)])(T)$, there exists a unique pair $(\alpha,\beta)$ of morphisms of schemes $\alpha : E \to E_{\mr{H}}$ and $\beta : T \to S_{\mr{H}}$ such that the diagram \begin{center}\begin{tikzpicture}[>=stealth] 
\matrix[matrix of math nodes,row sep=3em, column sep=1.5em, text height=1.5ex, text depth=0.25ex] { 
|[name=11]|  & |[name=12]| E & |[name=13]|  & |[name=14]| E_{\mr{H}} \\ 
|[name=21]| (\Z/(3))^{2}_{T} & |[name=22]|  & |[name=23]| (\Z/(3))^{2}_{S_{\mr{H}}} & |[name=24]|  \\ 
|[name=31]|  & |[name=32]| T & |[name=33]|  & |[name=34]| S_{\mr{H}} \\ 
}; 
\draw[->,font=\scriptsize]
(12) edge node[above=-1pt] {$\alpha$} (14) (32) edge node[below=-1pt] {$\beta$} (34) (21) edge[pos=0.8] node[above=-1pt] {$\id \times \beta$} (23) (21) edge node[above left=-2pt] {$\xi$} (12) (23) edge node[above left=-2pt] {$\xi_{\mr{H}}$} (14) (12) edge[pos=0.75] node[right=-1pt] {$f_{T}$} (32) (14) edge node[right=-1pt] {$f_{S_{\mr{H}}}$} (34) (21) edge (32) (23) edge (34); \end{tikzpicture} \end{center} commutes and induces an isomorphism of $T$-group schemes $E \simeq T \times_{\beta,S_{\mr{H}}} E_{\mr{H}}$ as in \labelcref{20-eqn-01}. \par As in \Cref{23}, for every $\sigma \in \GL_{2}(\Z/(3))$, let $\varphi_{\sigma}$ be the $S_{\mr{H}}$-automorphism of $(\Z/(3))^{2}_{S_{\mr{H}}}$ induced by $\sigma$; then precomposition $\xi_{\mr{H}} \varphi_{\sigma}$ defines another full level 3 structure on $E_{\mr{H}}/S_{\mr{H}}$. Taking $T = S_{\mr{H}}$ and $\xi = \xi_{\mr{H}}\varphi_{\sigma}$ above, there is a unique pair $(\alpha_{\sigma},\beta_{\sigma})$ of morphisms of schemes $\alpha_{\sigma} : E_{\mr{H}} \to E_{\mr{H}}$ and $\beta_{\sigma} : S_{\mr{H}} \to S_{\mr{H}}$ such that the diagram \begin{center}\begin{tikzpicture}[>=stealth] 
\matrix[matrix of math nodes,row sep=3em, column sep=1.5em, text height=1.5ex, text depth=0.25ex] { 
|[name=11]|  & |[name=12]| E_{\mr{H}} & |[name=13]|  & |[name=14]| E_{\mr{H}} \\ 
|[name=21]| (\Z/(3))^{2}_{S_{\mr{H}}} & |[name=22]|  & |[name=23]| (\Z/(3))^{2}_{S_{\mr{H}}} & |[name=24]|  \\ 
|[name=31]|  & |[name=32]| S_{\mr{H}} & |[name=33]|  & |[name=34]| S_{\mr{H}} \\ 
}; 
\draw[->,font=\scriptsize]
(12) edge node[above=-1pt] {$\alpha_{\sigma}$} (14) (32) edge node[below=-1pt] {$\beta_{\sigma}$} (34) (21) edge[pos=0.8] node[above=-1pt] {$\id \times \beta_{\sigma}$} (23) (21) edge node[above left=-2pt] {$\xi_{\mr{H}} \varphi_{\sigma}$} (12) (23) edge node[above left=-2pt] {$\xi_{\mr{H}}$} (14) (12) edge[pos=0.75] node[right=-1pt] {$f_{S_{\mr{H}}}$} (32) (14) edge node[right=-1pt] {$f_{S_{\mr{H}}}$} (34) (21) edge (32) (23) edge (34); \end{tikzpicture} \end{center} commutes and induces an isomorphism of $S_{\mr{H}}$-group schemes $E_{\mr{H}} \simeq S_{\mr{H}} \times_{\beta_{\sigma},S_{\mr{H}}} E_{\mr{H}}$. Given two elements $\sigma_{1},\sigma_{2} \in \GL_{2}(\Z/(3))$, we have a commutative diagram \begin{center}\begin{tikzpicture}[>=stealth] 
\matrix[matrix of math nodes,row sep=3em, column sep=1.5em, text height=1.5ex, text depth=0.25ex] { 
|[name=11]|  & |[name=12]| E_{\mr{H}} & |[name=13]|  & |[name=14]| E_{\mr{H}} & |[name=15]|  & |[name=16]| E_{\mr{H}} \\ 
|[name=21]| (\Z/(3))^{2}_{S_{\mr{H}}} & |[name=22]|  & |[name=23]| (\Z/(3))^{2}_{S_{\mr{H}}} & |[name=24]|  & |[name=25]| (\Z/(3))^{2}_{S_{\mr{H}}} & |[name=26]|  \\ 
|[name=31]|  & |[name=32]| S_{\mr{H}} & |[name=33]|  & |[name=34]| S_{\mr{H}} & |[name=35]|  & |[name=36]| S_{\mr{H}} \\ 
}; 
\draw[->,font=\scriptsize]
(12) edge node[above=-1pt] {$\alpha_{\sigma_{1}}$} (14) (14) edge node[above=-1pt] {$\alpha_{\sigma_{2}}$} (16) (32) edge node[below=-1pt] {$\beta_{\sigma_{1}}$} (34) (34) edge node[below=-1pt] {$\beta_{\sigma_{2}}$} (36) (21) edge (23) (23) edge (25) (21) edge node[above left=-2pt] {$\xi_{\mr{H}} \varphi_{\sigma_{1}} \varphi_{\sigma_{2}}$} (12) (23) edge node[above left=-2pt] {$\xi_{\mr{H}} \varphi_{\sigma_{2}}$} (14) (25) edge node[above left=-2pt] {$\xi_{\mr{H}}$} (16) (12) edge[pos=0.75] node[right=-1pt] {$f_{S_{\mr{H}}}$} (32) (14) edge[pos=0.75] node[right=-1pt] {$f_{S_{\mr{H}}}$} (34) (16) edge node[right=-1pt] {$f_{S_{\mr{H}}}$} (36) (21) edge (32) (23) edge (34) (25) edge (36); \end{tikzpicture} \end{center} which implies \[ \beta_{\sigma_{2}}\beta_{\sigma_{1}} = \beta_{\sigma_{1}\sigma_{2}} \] since $\varphi_{\sigma_{1}\sigma_{2}} = \varphi_{\sigma_{1}} \varphi_{\sigma_{2}}$ (see \Cref{23}). Thus the assignment \begin{align} \label{25-eqn-01} \sigma \mapsto \beta_{\sigma} \end{align} defines a right action of $\GL_{2}(\Z/(3))$ on the scheme $S_{\mr{H}}$. \par In terms of the generators \[ \mathsf{M}_{1} = \begin{bmatrix} 1 & 0 \\ 0 & -1 \end{bmatrix} \;,\; \mathsf{M}_{2} = \begin{bmatrix} 1 & 0 \\ -1 & 1 \end{bmatrix} \;,\; \mathsf{i} = \begin{bmatrix} 0 & -1 \\ 1 & 0 \end{bmatrix} \] of $\GL_{2}(\Z/(3))$, the action of $\GL_{2}(\Z/(3))$ on $E_{\mr{H}}/A_{\mr{H}}$ is as follows. (We refer to \labelcref{20150206-09-eqn-01} for the additive structure on $E_{\on{H}}[3]$.) \begin{enumerate} \itemsep5pt \item For $\sigma = \mathsf{M}_{1}$, the new level 3 structure $\xi_{\mr{H}}\varphi_{\mathsf{M}_{1}}$ is \[ \begin{bmatrix} [-1:0:1] & [1:-\omega:0] \end{bmatrix} \begin{bmatrix} 1 & 0 \\ 0 & -1 \end{bmatrix} = \begin{bmatrix} [-1:0:1] & [1:-\omega^{2}:0] \end{bmatrix} \] and the scheme morphisms $\alpha_{\mathsf{M}_{1}} : E_{\on{H}} \to E_{\on{H}}$ and $\beta_{\mathsf{M}_{1}} : S_{\mr{H}} \to S_{\mr{H}}$ correspond to the ring homomorphisms sending \begin{align*} \begin{cases} (X,Y,Z) \leftarrow\!\shortmid (X,Y,Z) \\ (\mu,\omega^{2}) \leftarrow\!\shortmid (\mu,\omega) \end{cases} \end{align*} respectively. \item For $\sigma = \mathsf{M}_{2}$, the new level 3 structure $\xi_{\mr{H}}\varphi_{\mathsf{M}_{2}}$ is \[ \begin{bmatrix} [-1:0:1] & [1:-\omega:0] \end{bmatrix} \begin{bmatrix} 1 & 0 \\ -1 & 1 \end{bmatrix} = \begin{bmatrix} [-\omega^{2}:0:1] & [1:-\omega:0] \end{bmatrix} \] and the scheme morphisms $\alpha_{\mathsf{M}_{2}} : E_{\on{H}} \to E_{\on{H}}$ and $\beta_{\mathsf{M}_{2}} : S_{\mr{H}} \to S_{\mr{H}}$ correspond to the ring homomorphisms sending \begin{align*} \begin{cases} (X,Y,\omega^{2} Z) \leftarrow\!\shortmid (X,Y,Z) \\ (\omega\mu,\omega) \leftarrow\!\shortmid (\mu,\omega) \end{cases} \end{align*} respectively. \item For $\sigma = \mathsf{i}$, the new level 3 structure $\xi_{\mr{H}}\varphi_{\mathsf{i}}$ is \[ \begin{bmatrix} [-1:0:1] & [1:-\omega:0] \end{bmatrix} \begin{bmatrix} 0 & -1 \\ 1 & 0 \end{bmatrix} = \begin{bmatrix} [1:-\omega:0] & [0:1:-1] \end{bmatrix} \] and the scheme morphisms $\alpha_{\mathsf{i}} : E_{\on{H}} \to E_{\on{H}}$ and $\beta_{\mathsf{i}} : S_{\mr{H}} \to S_{\mr{H}}$ correspond to the ring homomorphisms sending \begin{align*} \begin{cases} (\omega X + \omega^{2} Y + Z , \omega^{2} X + \omega Y + Z , X+Y+Z) \leftarrow\!\shortmid (X,Y,Z) \\ (\frac{\mu+2}{\mu-1},\omega) \leftarrow\!\shortmid (\mu,\omega) \end{cases} \end{align*} respectively. \end{enumerate} \end{pg}

\begin{remark} According to our convention, the action of $\GL_{2}(\Z/(3))$ on the fiber category $[\Gamma(3)](E_{\on{H}}/\Spec A_{\on{H}})$ is by precomposition, hence the action of $\GL_{2}(\Z/(3))$ on pairs of points on the right hand side of \labelcref{20150206-09-eqn-01} is a \emph{right} action; thus the induced action of $\GL_{2}(\Z/(3))$ on the scheme $\Spec A_{\on{H}}$ is a \emph{right} action (as described in \labelcref{25-eqn-01}) and the corresponding action of $\GL_{2}(\Z/(3))$ on the coordinate ring $A_{\on{H}}$ is a \emph{left} action. \end{remark}

\begin{pg}[Proof of {\Cref{24}}] In fact, it turns out that the identities \begin{align} \label{20150206-30-eqn-01} a_{1}^{3}-27a_{3} &= (3C+9B-1)^{3} \\ \label{20150206-30-eqn-02} a_{3} &= B(6C+9B-1) \end{align} hold in $A_{\on{W}}$ which yields a simpler description \begin{align*} A_{\on{W}} &\simeq \textstyle \Z[\frac{1}{3},B,C,\frac{1}{C} , \frac{1}{3C+9B-1} , \frac{1}{6C+9B-1}]/(C^{2}+3CB+3B^{2}) \end{align*} of $A_{\on{W}}$. (For \labelcref{20150206-30-eqn-01}, write out $a_{1}^{3}-27a_{3}$ in terms of $B,C$ and notice that it is of the form $9C+27B-1$ plus higher order terms; then check that the naive guess works. To see \labelcref{20150206-30-eqn-02}, substitute $C^{2} = -3CB-3B^{2}$ into \labelcref{20150206-30-eqn-04}.) \par We follow the argument of \cite[2.1]{ARTEBANI-DOLGACHEV-THPOPCC}; see also \cite[\S1.4.1, \S1.4.2]{CONNELL-ECH}. Working ``generically'', we will assume that $a_{1}$ is a unit to obtain the coordinate change formula \labelcref{20150206-30-eqn-09}, then observe that it applies also to the case when $a_{1}$ is not a unit. Starting with \begin{align} \label{20150206-30-eqn-05} Y_{1}Z_{1}(Y_{1} + a_{1}X_{1} + a_{3}Z_{1}) = X_{1}^{3} \end{align} we define $X_{2},Y_{2},Z_{2}$ by the system \[ \begin{bmatrix} X_{1} \\ Y_{1} \\ Z_{1} \end{bmatrix} = \begin{bmatrix} u^{2} & & \\ & u^{3} & \\ & & 1 \end{bmatrix} \begin{bmatrix} X_{2} \\ Y_{2} \\ Z_{2} \end{bmatrix} \] where $u = a_{1}/3$ and substitute into \labelcref{20150206-30-eqn-05} to get \begin{align} \label{20150206-30-eqn-06} \textstyle Y_{2}Z_{2}(Y_{2} + 3X_{2} + \frac{27a_{3}}{a_{1}^{3}}Z_{2}) = X_{2}^{3} \;. \end{align} We define $X_{3},Y_{3},Z_{3}$ by the system \[ \begin{bmatrix} 1 & 1 &  \\ 1 &  & \frac{27a_{3}}{a_{1}^{3}} \\ 1 &  &  \end{bmatrix} \begin{bmatrix} X_{2} \\ Y_{2} \\ Z_{2} \end{bmatrix} = \begin{bmatrix} \omega & \omega^{2} &  \\ \omega^{2} & \omega &  \\  &  & 1 \end{bmatrix} \begin{bmatrix} X_{3} \\ Y_{3} \\ Z_{3} \end{bmatrix} \] where $\omega = \frac{C+B}{B}$ \footnote{Since $3$ is invertible, if $x$ is a root of the polynomial $T^{2}+3T+3$ then $x+1$ is a root of the polynomial $T^{2}+T+1$, thus it is natural to take $\frac{C+B}{B}$ as our $\omega$.} and substitute into \labelcref{20150206-30-eqn-06} to get \[ \textstyle (\omega X_{3} + \omega^{2}Y_{3} - Z_{3})(\omega^{2} X_{3} + \omega Y_{3} - Z_{3})(-X_{3}-Y_{3}+Z_{3}) = \frac{27a_{3}}{a_{1}^{3}}Z_{3}^{3} \] or equivalently \begin{align} \label{20150206-30-eqn-07} \textstyle X_{3}^{3} + Y_{3}^{3} + \frac{27a_{3}-a_{1}^{3}}{a_{1}^{3}}Z_{3}^{3} = -3X_{3}Y_{3}Z_{3} \;. \end{align} We know that the coefficient of $Z_{3}^{3}$ in \labelcref{20150206-30-eqn-07} is a cube \labelcref{20150206-30-eqn-01} so we normalize by defining $X_{4},Y_{4},Z_{4}$ by the system \[ \begin{bmatrix} X_{3} \\ Y_{3} \\ Z_{3} \end{bmatrix} = \begin{bmatrix} 1 &  &  \\  & 1 &  \\  &  & \frac{-a_{1}}{3C+9B-1} \end{bmatrix} \begin{bmatrix} X_{4} \\ Y_{4} \\ Z_{4} \end{bmatrix} \] and substitute into \labelcref{20150206-30-eqn-07} to get \begin{align} \label{20150206-30-eqn-08} \textstyle X_{4}^{3}+Y_{4}^{3}+Z_{4}^{3} = 3\frac{a_{1}}{3C+9B-1}X_{4}Y_{4}Z_{4} \;. \end{align} To summarize the above, there is a ring homomorphism $\varphi_{21} : A_{\on{H}} \to A_{\on{W}}$ sending \begin{align} \label{20150206-30-eqn-13} \begin{aligned} &\mu \mapsto \frac{3C-1}{3C+9B-1} \\ &\omega \mapsto \frac{C+B}{B} \end{aligned} \end{align} and solving for $B,C$ in terms of $\mu,\omega$ implies that the inverse $\varphi_{12} : A_{\on{W}} \to A_{\on{H}}$ sends \begin{align} \label{20150206-30-eqn-14} \begin{aligned} B &\mapsto \frac{\mu-1}{3(\omega+2)(\mu-\omega)} \\ C &\mapsto \frac{(\omega-1)(\mu-1)}{3(\omega+2)(\mu-\omega)} \end{aligned} \end{align} where $\omega+2$ is a unit of $A_{\on{H}}$ since $(\omega+2)(\omega-1) = -3$ and $\mu-\omega$ is a unit of $A_{\on{H}}$ since $\mu^{3}-1 = (\mu-1)(\mu-\omega)(\mu-\omega^{2})$. We may check that the product \[ \begin{bmatrix} u^{2} & & \\ & u^{3} & \\ & & 1 \end{bmatrix} \begin{bmatrix} 1 & 1 &  \\ 1 &  & \frac{27a_{3}}{a_{1}^{3}} \\ 1 &  &  \end{bmatrix}^{-1} \begin{bmatrix} \omega & \omega^{2} &  \\ \omega^{2} & \omega &  \\  &  & 1 \end{bmatrix} \begin{bmatrix} 1 &  &  \\  & 1 &  \\  &  & \frac{-a_{1}}{3C+9B-1} \end{bmatrix} \] is ``projectively equivalent'' to the matrix \begin{align} \label{20150206-30-eqn-09} X := \begin{bmatrix} 0 & 0 & \frac{-3}{3C+9B-1} \\[4pt] \omega & \omega^{2} & \frac{3u}{3C+9B-1} \\[4pt] \frac{\omega^{2}}{a_{3}} & \frac{\omega}{a_{3}} & \frac{3u}{a_{3}(3C+9B-1)} \end{bmatrix} \end{align} whose determinant is a unit of $A_{\on{W}}$. Given a section $[s_{X}:s_{Y}:s_{Z}]$ of \labelcref{20150206-30-eqn-05}, the corresponding section of \labelcref{20150206-30-eqn-08} is $X^{-1} \cdot [s_{X}:s_{Y}:s_{Z}]^{\on{T}}$ where \[ X^{-1} = \begin{bmatrix} \frac{-a_{1}}{3} & \frac{B}{C} & \frac{-9CB-18B^{2}-C}{3} \\[4pt] \frac{-a_{1}}{3} & \frac{-B}{C+3B} & \frac{-9CB-9B^{2}+C+3B}{3} \\[4pt] \frac{-3C-9B+1}{3} & 0 & 0 \end{bmatrix} \;. \] The above implies that the sections \[ [0:1:0] \;,\; [0:0:1] \;,\; [C:B+C:1] \] of \labelcref{20150206-30-eqn-05} (i.e. the identity section and ordered basis for the 3-torsion) correspond to the sections \begin{align} \label{20150206-30-eqn-10} [1:-\omega:0] \;,\; [1:-\omega^{2}:0] \;,\; [-1:0:1] \end{align} of \labelcref{20150206-30-eqn-08}. We may apply an automorphism of the pair $(A_{\on{H}},E_{\on{H}}/A_{\on{H}}) \in \ms{M}_{1,1,\Z}$ of the form \labelcref{25}(2) (for $Y$ instead of $Z$) to \labelcref{20150206-30-eqn-10} to get \begin{align} \label{20150206-30-eqn-11} [1:-1:0] \;,\; [1:-\omega:0] \;,\; [-1:0:1] \end{align} and using the fact that there is a simply transitive action of $\GL_{2}(\Z/(3))$ on the set of ordered bases of the 3-torsion in $E_{\on{H}}/A_{\on{H}}$, we may switch the second and third sections of \labelcref{20150206-30-eqn-11} to obtain \begin{align} \label{20150206-30-eqn-12} [1:-1:0] \;,\; [-1:0:1] \;,\; [1:-\omega:0] \end{align} as desired. \qed \end{pg}

\begin{remark} For \labelcref{20150206-30-eqn-01}, see also Stojanoska's derivation \cite[\S4.1]{STOJANOSKA-CDF2PTMF}. \end{remark}

\begin{remark} There are coordinate change formulas in \cite[\S3]{SMART-HESSIAN} transforming a Weierstrass equation into Hesse normal form, but there it is assumed that the base ring is a finite field $\F_{q}$ where $q \equiv 2 \pmod{3}$, in order to take cube roots of $a_{1}^{3}-27a_{3}$, but from this description it is not clear that the cube root is an algebraic function. As shown in \labelcref{20150206-30-eqn-01}, it turns out that in fact $a_{1}^{3}-27a_{3}$ is a cube in the ring $A_{\on{W}}$. One suspects that this is the case after tracing through the proof of \cite[2.1]{ARTEBANI-DOLGACHEV-THPOPCC} and arriving at the equation $x^{3} + y^{3} + \frac{27a_{3}-a_{1}^{3}}{a_{1}^{3}}z^{3} = 3xyz$, in which case we know that $\frac{27a_{3}-a_{1}^{3}}{a_{1}^{3}}$ is a cube by \Cref{20150206-28}. \end{remark}

\begin{lemma} \label{20150206-28} Let $k$ be a field of characteristic not $3$, and let \begin{align} \label{20150206-28-eqn-01} x^{3} + y^{3} + \beta = 3xy \end{align} be a curve in $\A_{k}^{2}$. Suppose that \begin{align} \label{20150206-28-eqn-02} ax + by + c = 0 \end{align} is the tangent line to a flex point of $E$ and suppose that $a^{3} \ne b^{3}$. Then $\beta$ is a cube in $k$. \end{lemma} \begin{proof} If $a = 0$, then $b \ne 0$ and substituting $y = -\frac{c}{b}$ into \labelcref{20150206-28-eqn-01} and rearranging gives $x^{3}+\frac{3c}{b}x - (\frac{c}{b})^{3}+\beta = 0$ which by assumption is of the form $(x+\ell)^{3}$ for some $\ell \in k$. Comparing coefficients, we have $\ell = 0$ and so $\beta = (\frac{c}{b})^{3}$. \par By symmetry we may assume that $a,b \ne 0$. By scaling \labelcref{20150206-28-eqn-02}, we may assume that $b = -1$. Substituting $y=ax+c$ into $E$ gives \[ (a^{3}+1)x^{3} + 3(a)(ac-1)x^{2} + 3(c)(ac-1)x + (c^{3}+\beta) \] and dividing by the leading coefficient gives \[ x^{3} + 3\left(\frac{a(ac-1)}{a^{3}+1}\right)x^{2} + 3\left(\frac{c(ac-1)}{a^{3}+1}\right)x + \left(\frac{c^{3}+\beta}{a^{3}+1}\right) \] and comparing this to \[ x^{3}+3\ell x^{2} + 3\ell^{2}x + \ell^{3} \] gives either $ac-1 = 0$ in which case $c^{3}+\beta = 0$ as well (so that $\beta = (-1/a)^{3} = (-c)^{3}$), otherwise if $ac-1 \ne 0$ then \[ \frac{c}{a} = a \left(\frac{ac-1}{a^{3}+1}\right) \] which implies $c = -a^{2}$ so that the original equation of the tangent line is $y = ax-a^{2}$. Substituting this back into $E$ gives $\beta = (-a)^{3}$. \end{proof}

\section{Higher direct images of sheaves on classifying stacks of discrete groups} \label{sec-06}

The material in this section is standard and we claim no originality.

For a category $\mc{C}$, we denote by $\on{PSh}(\mc{C})$ (resp. $\on{PAb}(\mc{C})$) the category of presheaves (resp. abelian presheaves) on $\mc{C}$. If $\mc{C}$ is a site, we denote by $\on{Sh}(\mc{C})$ (resp. $\on{Ab}(\mc{C})$) the category of sheaves (resp. abelian sheaves) on $\mc{C}$. 

Let $\mc{C}$ be a site, let $G$ be a finite (discrete) group, let $\mathrm{B}G_{\mc{C}}$ be the classifying stack associated to $G$ over $\mc{C}$. Let \[ \pi : \mathrm{B}G_{\mc{C}} \to \mc{C} \] be the projection and let \[ \varphi : \mc{C} \to \mathrm{B}G_{\mc{C}} \] be the canonical section of $\pi$. We view any fibered category $p : \mc{F} \to \mc{C}$ as a site via the Grothendieck topology inherited from $\mc{C}$ via $p$.

\begin{lemma} \label{45} In the setup above, for any abelian sheaf $\ms{F} \in \on{Ab}(\mathrm{B}G_{\mc{C}})$ the higher pushforward $\mathbf{R}^{i}\pi_{\ast}\ms{F}$ is naturally isomorphic to the sheaf associated to the presheaf whose value on an object $U \in \mc{C}$ is $\H^{i}(G,\Gamma(U,\varphi^{\ast}\ms{F}))$. \end{lemma} \begin{proof} Let $\mathrm{P}G_{\mc{C}}$ denote the category whose objects are the objects of $\mc{C}$ and where a morphism $X_{1} \to X_{2}$ in $\mathrm{P}G_{\mc{C}}$ is a pair $(\varphi,g)$ where $\varphi \in \on{Mor}_{\mc{C}}(X_{1},X_{2})$ and $g \in G$. (In other words, there is an equivalence of categories $\mathrm{P}G_{\mc{C}} \simeq \mc{C} \times [\ast/G]$ where $[\ast/G]$ is the category with one object $\ast$ and where $\Hom_{[\ast/G]}(\ast,\ast)$ is isomorphic to $G$.) The fibered category $\mathrm{P}G_{\mc{C}}$ is a (separated) prestack whose associated stack is $\mathrm{B}G_{\mc{C}}$, and the inclusion $\mathrm{P}G_{\mc{C}} \to \mathrm{B}G_{\mc{C}}$ induces an equivalence of topoi $\on{Sh}(\mathrm{P}G_{\mc{C}}) \simeq \on{Sh}(\mathrm{B}G_{\mc{C}})$. Hence in the statement of the lemma we may replace $\mathrm{B}G_{\mc{C}}$ by $\mathrm{P}G_{\mc{C}}$ where by abuse of notation we also denote \[ \pi : \mathrm{P}G_{\mc{C}} \to \mc{C} \] the projection morphism. Since sheafification is an exact functor, the diagram \begin{center}\begin{tikzpicture}[>=stealth] 
\matrix[matrix of math nodes,row sep=2em, column sep=2em, text height=1.5ex, text depth=0.25ex] { 
|[name=11]| \on{PAb}(\mathrm{P}G_{\mc{C}}) & |[name=12]| \on{PAb}(\mc{C}) \\ 
|[name=21]| \on{Ab}(\mathrm{P}G_{\mc{C}}) & |[name=22]| \on{Ab}(\mc{C}) \\
}; 
\draw[->,font=\scriptsize]
(11) edge node[above=-1pt] {$\pi^{\on{pre}}_{\ast}$} (12) (21) edge node[below=-1pt] {$\pi_{\ast}$} (22) (11) edge node[left=-1pt] {$\on{sh}$} (21) (12) edge node[right=-1pt] {$\on{sh}$} (22); \end{tikzpicture} \end{center} is (2-)commutative. For the same reason, we have a natural isomorphism \begin{align} \label{45-eqn-01} (\mathbf{R}\pi^{\on{pre}}_{\ast}(\ms{F}))^{\on{sh}} \simeq \mathbf{R}\pi_{\ast}(\ms{F}^{\on{sh}}) \end{align} in $\on{D}^{+}(\on{Ab}(\mc{C}))$ for any abelian presheaf $\ms{F} \in \on{PAb}(\mathrm{P}G_{\mc{C}})$. Presheaves on $\mathrm{P}G_{\mc{C}}$ correspond to presheaves $\ms{F}$ on $\mc{C}$ equipped with a $G$-action, and under this identification $\pi^{\on{pre}}_{\ast}(\ms{F}) = \ms{F}^{G}$ where $\Gamma(U,\ms{F}^{G}) := (\Gamma(U,\ms{F}))^{G}$ for all $U \in \mc{C}$. Let $\ms{F} \in \on{Ab}(\mathrm{P}G_{\mc{C}})$ be an abelian sheaf, and let \[ \ms{F} \to \mc{I}^{0} \to \mc{I}^{1} \to \mc{I}^{2} \to \dotsb \] be a resolution of $\ms{F}$ by injective abelian presheaves $\mc{I}^{i} \in \on{PAb}(\mathrm{P}G_{\mc{C}})$. Then $\mathbf{R}\pi^{\on{pre}}_{\ast}(\ms{F})$ is isomorphic to \begin{align} \label{45-eqn-02} (\mc{I}^{\bullet})^{G} = \{(\mc{I}^{0})^{G} \to (\mc{I}^{1})^{G} \to (\mc{I}^{2})^{G} \to \dotsb\} \end{align} in $\on{D}^{+}(\on{PAb}(\mc{C}))$, and $\Gamma(U,\mathbf{R}\pi^{\on{pre}}_{\ast}(\ms{F}))$ is isomorphic to \begin{align} \label{45-eqn-03} \Gamma(U,(\mc{I}^{\bullet})^{G}) = \{(\Gamma(U,\mc{I}^{0}))^{G} \to (\Gamma(U,\mc{I}^{1}))^{G} \to (\Gamma(U,\mc{I}^{2}))^{G} \to \dotsb\} \end{align} in $\on{D}^{+}(\on{PAb}(\mc{C}))$. Furthermore $\Gamma(U,\mc{I}^{i}) \simeq (i_{U})^{\ast}\mc{I}^{i}$ is an injective $G$-module for all $i$ by \Cref{46}, thus we have an isomorphism \[ h^{i}(\Gamma(U,(\mc{I}^{\bullet})^{G})) \simeq \H^{i}(G,\Gamma(U,\ms{F})) \] of abelian groups. \end{proof}

\begin{lemma} \label{46} Let $\mc{C}$ be a category, let $U \in \mc{C}$ be an object, let $\mc{A}_{\mc{C},U}$ denote the full subcategory of $\mc{C}$ containing exactly $U$, and let $i_{U} : \mc{A}_{\mc{C},U} \to \mc{C}$ denote the inclusion. The inverse image functor $(i_{U})^{\ast} : \on{PAb}(\mc{C}) \to \on{PAb}(\mc{A}_{\mc{C},U})$ preserves injectives. \end{lemma} \begin{proof} The functor $(i_{U})^{\ast} : \on{PAb}(\mathrm{P}G_{\mc{C}}) \to \on{PAb}(\mc{A}_{\mc{C},U})$ has an exact left adjoint, namely the ``extension by zero'' functor $i_{U,!} : \on{PAb}(\mc{A}_{\mc{C},U}) \to \on{PAb}(\mathrm{P}G_{\mc{C}})$ which sends $M \in \on{PAb}(\mc{A}_{\mc{C},U})$ to the abelian presheaf $i_{U,\dagger}(M)$ where $\Gamma(V,i_{U,\dagger}(M)) = M$ if $V = U$ and $0$ otherwise (with the only nontrivial restriction morphisms being those corresponding to the endomorphisms of $U$).  \end{proof}

\section{Computation using Magma} \label{sec-07}

We compute $\H^{1}(\GL_{2}(\Z/(3)),M)$ in \Cref{04-03} using \textsc{Magma} \cite{BOSMA1997235}. Here \texttt{G} is defined as the subgroup of $\GL_{2}(\Z/(3))$ generated by the matrices in \labelcref{04-eqn-08}, but the specified matrices constitute a generating set so in fact \texttt{G} $ = \GL_{2}(\Z/(3))$. The group \texttt{G} acts on the abelian group $M = (\Z/(2))^{\oplus 6}$ by the three specified elements of $\on{Mat}_{6 \times 6}(\Z)$, where each $\mb{x} \in M$ is viewed as a horizontal vector and each $6 \times 6$ matrix $\mb{A}$ acts on $M$ by right multiplication $\mb{x} \mapsto \mb{x} \cdot \mb{A}$. The last line computes $\H^{1}(G,(\Z/(2))^{\oplus 6})$.

{\small
\begin{verbatim}
G := MatrixGroup< 2 , FiniteField(3) |
  [ 1,0 , -1,1 ] , [ 0,-1 , 1,0] , [ 1,0 , 0,-1 ] 
>;
mats := [
  Matrix(Integers() , 6 , 6 , [
    0, 0, 1, 0, 0, 0 ,
    1, 0, 0, 0, 0, 0 ,
    0, 1, 0, 0, 0, 0 ,
    0, 0, 0, 0, 1, 0 ,
    0, 0, 0, 0, 0, 1 ,
    0, 0, 0, 1, 0, 0 ]) ,
  Matrix(Integers() , 6 , 6 , [
    1, 0, 0, 0, 0, 0 ,
    1, 0, 1, 0, 0, 0 ,
    1, 1, 0, 0, 0, 0 ,
    0, 0, 0, 1, 0, 0 ,
    0, 0, 0, 1, 0, 1 ,
    0, 0, 0, 1, 1, 0 ]) ,
  Matrix(Integers() , 6 , 6 , [
    0, 0, 0, 1, 0, 0 ,
    0, 0, 0, 0, 1, 0 ,
    0, 0, 0, 0, 0, 1 ,
    1, 0, 0, 0, 0, 0 ,
    0, 1, 0, 0, 0, 0 ,
    0, 0, 1, 0, 0, 0 ])
];
CM := CohomologyModule(G,[2,2,2,2,2,2],mats);
CohomologyGroup(CM,1);
\end{verbatim}
}

\bibliography{../allbib.bib}
\bibliographystyle{alpha}

\end{document}